\documentclass[a4paper, 11pt]{article}

\usepackage{geometry}
\usepackage{amsthm, amsmath, mathtools, tensor, amssymb, upgreek}
\usepackage{booktabs}
\usepackage{multirow}
\usepackage{array}
\usepackage{latexsym}
\usepackage{float}
\usepackage{diagbox}
\usepackage{threeparttable}

\usepackage[T1]{fontenc}

\usepackage[usenames]{color}
\usepackage[colorlinks=true]{hyperref}
\usepackage[title]{appendix}

\definecolor{mygray}{gray}{0.9}
\definecolor{deeppink}{RGB}{255,20,147}
\definecolor{mygreen}{rgb}{0.05, 0.576, 0.03}
\definecolor{myred}{rgb}{0.768, 0.09, 0.09}

\usepackage{caption}
\usepackage{cite}
\newtheorem{theorem}{Theorem}[section]
\newtheorem{proposition}[theorem]{Proposition}

\newtheorem{corollary}[theorem]{Corollary}
\newtheorem{remark}[theorem]{Remark}

\newtheorem{example}[theorem]{Example}

\newcommand{\Q}{\mathsf{Q}}

\newcommand{\R}{\mathbb{R}}
\newcommand{\E}{\textsf{E}}
\newcommand{\dd}{\mathrm{d}}
\newcommand{\ii}{{\rm i}}
\newcommand{\PP}{\textsf{P}}

\newcommand{\Var}{\textsf{Var}}

\newcommand{\Gf}{\mathrm{\Gamma}}
\newcommand{\Li}{\mathrm{Li}}

\renewcommand{\Re}{\mathrm{Re}}

\allowdisplaybreaks

\numberwithin{equation}{section}
\numberwithin{equation}{subsection}
\numberwithin{equation}{subsubsection}

\renewcommand{\theequation}{%
  \ifnum\value{subsubsection}>0
    \thesubsubsection.\arabic{equation}%
  \else\ifnum\value{subsection}>0
    \thesubsection.\arabic{equation}%
  \else
    \thesection.\arabic{equation}%
  \fi\fi}

\usepackage{graphicx}

\renewcommand{\Phi}{\varPhi}
\renewcommand{\Lambda}{\varLambda}
\renewcommand{\Psi}{\varPsi}
\renewcommand{\Omega}{\varOmega}

\begin{document}

\title{\bf Density estimation for exponential functionals \\ of subordinators}
\author{Dongdong Hu\thanks{Yiwu Industrial \& Commercial College, Yiwu, China. Email: \underline{hudongdong@ywicc.edu.cn}} \and Hasanjan Sayit\thanks{Xi'an Jiaotong Liverpool University, Suzhou, China. Email: \underline{Hasanjan.Sayit@xjtlu.edu.cn} }
\and Weixuan Xia\thanks{University of Central Florida, Orlando, US. Email: \underline{Weixuan.Xia@ucf.edu}}}

\date{2026}

\maketitle

\begin{abstract}
  Exponential functionals of subordinators play a key role in probability theory and arise in numerous applications, yet explicit density formulas have until now been available only in a limited number of special settings. This paper provides a general remedy. First, we obtain a closed-form representation of the density for the entire class of integer-valued subordinators (compound Poisson processes with positive integer jump sizes) in the form of a rapidly converging series whose coefficients obey a simple recurrence; we further extend the construction to include drift by means of an explicit iterative scheme. Second, we propose a general approximation method for arbitrary driftless subordinators, including cases where the L\'{e}vy measure has continuous support. By discretizing the L\'{e}vy measure, we generate a sequence of explicit, computable densities that is proved to converge to the true density. Together, these results greatly broaden the theory of exponential functionals, transforming a problem previously tractable only in rare instances into one that is either explicitly solvable or can be systematically approximated for the full class of subordinators. The resulting formulas enable fast and accurate computations for financial modelling, insurance risk assessment, and stochastic simulation, while the approximation procedure provides a robust and stable alternative to existing numerical methods. The unified framework --- linking pantograph equations, $q$-calculus, and Poisson approximation --- also opens avenues for analyzing more general decreasing functionals and drifted subordinators with general L\'{e}vy measures. \medskip
\end{abstract}

\medskip

\section{Introduction}\label{S:1}

The exponential functional of a subordinator, defined as the random integral
\[
I_{[X]}=\int_0^{\infty}e^{-X_t}\dd t,
\]
is a deceptively simple object that lies at the confluence of probability theory, stochastic analysis, and applied mathematics. For a real-valued L\'evy process $X$, this integral represents the total discounted occupation time of the process and encodes profound information about its sample path structure. The ubiquity of such functionals is matched only by the diversity of their applications: They govern the absorption times of self-similar Markov processes via the Lamperti transform (see Patie \cite{P12}), determine the stationary distributions of generalized Ornstein--Uhlenbeck processes (see Erickson and Maller \cite{EM05}), underpin the pricing of path-dependent derivatives in jump-diffusion models (see Cai and Kou \cite{CK12}), and appear in the specification of stochastic volatility and low-latency financial econometrics (see Barndorff-Nielsen \text{et al.} \cite{BPS12}). In statistical physics, these objects serve to represent the partition function of disordered systems, while in insurance mathematics they are used to compute ruin probabilities for heavy-tailed claim sizes. They also arise in applied network modeling: Guillemin et al. \cite{fab-rob-zwart} demonstrated that the asymptotic behavior of AIMD/TCP congestion windows under clustered losses can be formulated in terms of these functionals and that $q$-calculus yields various explicit expressions for the associated moments and densities.

Despite their fundamental importance and decades of intensive study, the distributional properties of exponential functionals have remained remarkably resistant to explicit characterization. The conditions for almost-sure finiteness are now well understood through the seminal works of Bertoin and Yor \cite{BY05} and Bertoin \text{et al.} \cite{BLM08}, and the absolute continuity of their laws has been established by Carmona \text{et al.} \cite{CPY97} for all non-trivial driftless subordinators. However, the explicit determination of the density function --- a problem that would seem fundamental for any applied usage --- has proven to be a formidable challenge. For the vast majority of subordinators, the density of $I_{[X]}$ is only known through complex inversion formulas: Mellin transforms requiring numerical integration in the complex plane \cite{P12}, Fourier inversion techniques that struggle with slow decay, or asymptotic expansions valid only in limiting regimes \cite{KPS12, MZ06, MS23}. Explicit closed-form densities have been obtained only for a handful of special processes including the Poisson process \cite{BBY04, BY05}, certain stable processes \cite{CKP09}, and hypergeometric L\'{e}vy processes \cite{K12}. This stark contrast between the universality of the object and the paucity of explicit formulas constitutes a significant gap in the theory, one that limits both theoretical understanding and practical computation.

The present paper bridges this gap through two major contributions that together provide a comprehensive treatment of exponential functional densities across the entire subordinator family, from exactly solvable cases to general approximations with rigorous convergence guarantees.

Our first contribution is to establish, for the entire class of integer-valued subordinators (IVSs), an explicit expression for the density of their exponential functionals. An IVS is a compound Poisson process whose jumps take values in the positive integers; this class is surprisingly broad and encompasses a wide array of processes of applied interest, including Poisson processes, multiple iterated Poisson processes (MIPPs) (see Hu \text{et al.} \cite{H26}), which arise naturally in models of catastrophic risk and low-latency finance (see Barndorff-Nielsen \text{et al.} \cite{BPS12}), space-fractional Poisson processes (see Orsingher and Polito \cite{OP12A}), and negative-binomial L\'{e}vy processes (see Xia \cite{X20}). Our main result in this direction, Theorem \ref{main}, demonstrates that for any driftless IVS, the density of $I_{[X]}$ admits a representation as an absolutely convergent generalized Dirichlet series:
\[
\phi_q(x)=\frac{\lambda \PP\{Z_1\geq 1\}}{\sum_{j=0}^{\infty}c_jq^j}\sum_{j=0}^{\infty}c_je^{-\lambda \PP\{Z_1\geq 1\}q^{-j}x}, \; \; x>0,
\]
where $q\in (0, 1)$ is the base parameter of the functional, $\lambda>0$ is the jump intensity, $Z_1$ is the generic jump size of the compound Poisson process, and the coefficients $c_j$ are determined by a simple recurrence relation.

This result substantially extends the known cases (which are confined to a few specific L\'{e}vy processes) to the entire IVS class, thus providing a unified framework that subsumes and generalizes all previous explicit formulas. Moreover, the series representation is not merely an abstract existence result: It is computationally efficient, numerically stable, and yields explicit formulas for the cumulative distribution function and the Laplace transform, as demonstrated in Corollary \ref{cor:cdf} and the subsequent remarks. The proof of this result is rooted in the theory of pantograph equations (functional-differential equations of advanced type), which we solve using a combination of generalized Dirichlet series and unilateral Z-transforms. This approach, building on the work of Iserles \cite{I93}, provides a global analytical framework for handling the non-local structure of the Fokker--Planck equation (\ref{2:k}) that characterizes the density.

In the presence of a positive drift $\mu\geq 0$,  the exponential functional $I_{[X]}$ has bounded support $(0, 1/\mu_q),$ where $\mu_q=(-\log q)\mu$. The density vanishes outside this interval, and its behavior on the interior is significantly more complex due to a singular term $(1-\mu_qx)$ in the Fokker--Planck equation. Nevertheless, we show in Theorem \ref{side} that the density can be constructed piecewise through a constructive iterative scheme. On each interval of the form $(q^{j+1}/\mu_q, q^j/\mu_q)$, $j\geq0$, the density is represented by basis functions $h_{q, j}$'s determined recursively through first-order linear ordinary differential equations, with continuity conditions at the boundary points. While this does not yield a single closed-form expression, it provides a systematic, high-precision numerical procedure that is far more reliable than generic numerical inversion techniques. We illustrate the utility of this scheme in Example \ref{e:pdu} and Example \ref{e:mippd}, and we provide numerical implementations in Section \ref{S:4}.

Our second contribution is a universal approximation framework for the density of exponential functionals of arbitrary driftless subordinators, including those with L\'{e}vy measures supported on continuous sets, such as tempered stable subordinators, gamma processes, and inverse Gaussian processes. For such processes, exact closed-form densities are generally not available, and existing numerical methods often suffer from instability, slow convergence, or lack of rigorous error bounds. Our approach is both constructive and theoretically sound: We discretize the L\'{e}vy measure of the target subordinator $X$ by quantizing jump sizes onto a lattice of spacing $\epsilon>0$, yielding a sequence of compound Poisson processes $X^{(\epsilon)}$ with integer-valued jumps for which the densities are explicitly computable via our IVS formulas. We then prove that the density of the exponential functional of $X^{(\epsilon)}$ converges pointwise to the density of exponential functional of $X$.

This result is far from trivial as convergence in distribution does not generally imply convergence of densities. We establish pointwise convergence through a delicate argument involving moment convergence, uniform integrability, and the uniqueness of probability measures determined by their moments. The approximating densities $\phi_{[X^{(\epsilon)}]}$ are given explicitly in the same generalized Dirichlet series form as (\ref{4.7}), with coefficients determined by the discretized L\'{e}vy measure. Thus, our approximation scheme yields a constructive limit representation for
$\phi_{[X]}$ that can be implemented numerically for any subordinator, with rigorous control over the approximation error.

To the best of our knowledge, this is the first universal approximation scheme for exponential functional densities that provides both explicit formulas and rigorous convergence guarantees. It complements and, in many cases, improves upon existing approaches: While the Mellin--Barnes representation of Patie and Savov \cite{PS18} provides a theoretical expression for the density, its numerical implementation requires contour integration and suffers from the evaluation of complex special functions; our approximation, by contrast, reduces the problem to summing elementary exponential functions and computing a recurrence relation for the coefficients, which is both straightforward and numerically stable.

The methodological innovations underlying our results draw from several distinct areas of mathematics. The theory of pantograph equations by Iserles \cite{I93} provides the analytical machinery for solving the Fokker--Planck equation in the IVS case; the use of Z-transforms and $q$-calculus by Gasper and Rahman \cite{GR90} yields the recurrence relations for the coefficients; and the discretization argument for general subordinators leverages deep results on Poisson random measures and the moment-determinacy of exponential functionals by Bertoin and Yor \cite{BY05}. The combination of these techniques is, as we believe, novel and opens the door to the analysis of more general decreasing functionals, such as the inverse-power functionals considered in Xia \cite{X22}, and to the extension of our approximation scheme to general two-sided L\'{e}vy processes.

The remainder of the paper is organized as follows. Section \ref{S:2} presents the necessary background, including the definition and basic properties of subordinators and their exponential functionals (Subsection \ref{S:2:1}), a detailed description of IVSs together with several central examples (Subsection \ref{S:2:2}). Section \ref{S:3} contains the main results: Subsection \ref{S:3.1} establishes an explicit density expression for driftless IVSs (Theorem \ref{main}) and examines consequences such as the cumulative distribution function (Corollary \ref{cor:cdf}) and applications to specific processes; Subsection \ref{S:3.2} extends the analysis to IVSs with strictly positive drift, proposing a constructive iterative method (Theorem \ref{side}) and examples illustrating the emergence of polylogarithms and hypergeometric functions. Section \ref{S:41} introduces a general approximation framework for exponential functionals of driftless subordinators with continuous L\'{e}vy measures, proving pointwise convergence of the approximating densities (Theorem \ref{main2}) and providing explicit convergence rates (Corollary \ref{cor:main}). Section \ref{S:4} offers extensive numerical illustrations and specification studies, including density plots for MIPPs, space-fractional Poisson processes, negative-binomial processes, and their drifted counterparts, as well as benchmark comparisons with known distributions. Lastly, Section \ref{S:6} concludes with final remarks and outlines directions for future research, such as extensions to drifted subordinators with continuous L\'{e}vy measures and the study of more general decreasing functionals.

We end the introduction by stressing that the paper's contributions are not only of theoretical interest, as they yield practical, implementable methods for computing distributions of exponential functionals in many applied settings. The explicit IVS formulas support fast and accurate calculations in financial modelling, insurance risk, and stochastic simulation, as the approximation strategy for general subordinators provides a dependable alternative to existing numerical techniques, while the overarching analytical framework --- linking pantograph equations, $q$-calculus, and Poisson approximation --- serves as a template for tackling a wider range of functional equations in probability theory. We hope that these methods and findings will encourage further developments and applications across the mathematical sciences.

\medskip

\section{Preliminaries}\label{S:2}
\subsection{Subordinators and their exponential functionals}\label{S:2:1}

We begin by recalling the fundamental objects of this paper.
A subordinator is a real-valued L\'evy process $X=(X_t)_{t\geq 0}$ with nondecreasing sample paths, starting at $X_0=0$. By the L\'evy--Khintchine formula, the Laplace transform of $X$ is given by
\begin{equation*}
  \E e^{-u X_t} = e^{-t \Psi(u)}, \quad \Re u \geq 0,
\end{equation*}
where the Laplace exponent $\Psi$ is nonnegative-valued and admits the representation
\begin{equation*}
  \Psi(u) = \mu u+\int_{0+}^{\infty} (1 - e^{-u z}) \nu(\dd z).
\end{equation*}
Here, $\mu\geq 0$ is the drift coefficient and $\nu$ is a L\'evy measure satisfying
\begin{equation*}
 \int^{\infty}_{0+}(1\wedge z)\nu(\dd z)<\infty.
\end{equation*}
The L\'evy measure $\nu$ governs the intensity and size of jumps: For any Borel set $A\subset (0, \infty)$, $\nu(A)$ is the expected number of jumps of $X$ with magnitude in $A$ per unit time. We denote the tail function of the L\'evy measure by
\begin{equation*}
\bar{\nu}(z):=\nu((z, \infty)), \; \; z>0.
\end{equation*}
A subordinator is said to be driftless if $\mu=0$. In what follows, unless stated otherwise, our attention is mainly restricted to driftless subordinators. Throughout, $X$ denotes a driftless subordinator started from zero, with no killing. Its L\'evy measure is written as $\nu_X$, and we denote the tail by $\bar{\nu}_X(z):=\nu_X((z,\infty))$, for $z>0$. The Laplace exponent of $X$ is given by
\begin{equation*}
  \Psi(u)=\int_{0+}^{\infty} (1-e^{-u z})\,\nu_X(\dd z).
\end{equation*}
Here, $\Psi$ is a Bernstein function and satisfies
\begin{equation*}
\Psi(0)=0, \; \; \lim_{u\to\infty}\Psi(u)=\nu((0,\infty))\subset (0,\infty],
\end{equation*}
together with $(-1)^{n-1}\Psi^{(n)}(u)\ge 0$ for every $n\ge 1$. In addition, $\Psi(u)/u$ is completely monotone, namely $(-1)^n \dd^n(\Psi(u)/u)/\dd u^n\ge 0$ for all $n\ge 0$.

The exponential functional associated with a driftless subordinator $X$ is defined by
\begin{equation}\label{ex-func}
I_{[X]}=\int_{0}^{\infty} e^{-X_t}\,\dd t.
\end{equation}
This random integral is finite $\PP$-\text{a.s.} if $X_t/\log t\overset{\rm a.s.}{\to}\infty$ as $t\to\infty$, which is equivalent to the conditions $\int_0^1 \bar{\nu}(z)\dd z<\infty$ and $\int_1^{\infty}\bar{\nu}(z)/z\,\dd z<\infty$; see Bertoin and Yor \cite[Theorem 1]{BY05}. From now on, we implicitly assume that these requirements are satisfied.

A key result of Carmona \text{et al.} \cite[Proposition 2.1]{CPY97} shows that, for any nontrivial driftless subordinator, the law of \(I_{[X]}\) is absolutely continuous on \((0,\infty)\), and its density \(\phi_{[X]}\) solves the integral equation
\begin{equation}\label{2:k}
  k(x) = \int_x^{\infty} \bar{\nu}_X\Big(\log\frac{y}{x}\Big) \, k(y) \, \dd y.
\end{equation}
As mentioned in the introduction, a closed-form expression for \(k(x)\) is available only for a few exceptional choices of driftless subordinators \(X\).
In the absence of general closed-form formulas, a substantial part of the earlier work concentrates on identifying the precise asymptotic behavior of the density as \(x\to\infty\). Minchev and Savov \cite{MS23} obtained nonclassical Tauberian asymptotics for the tail, the density, and its derivatives for a wide family of exponential functionals of subordinators. Their approach combines a saddle-point analysis with the Mellin transform, expressed via Bernstein-gamma functions, and requires only a relatively mild positive-increase assumption on the L\'evy measure. They further developed Stirling-type expansions for Bernstein-gamma functions in the complex setting, enabling them to show that many of these densities admit an analytic continuation to a conic domain in the complex plane. Haas \cite{Haas} proved sharp large-time asymptotic equivalents for both the density and the upper-tail probability in terms of the subordinator's Laplace exponent, improving earlier estimates that were accurate only at the logarithmic level. In addition, this study provides a direct connection between the behavior of the L\'evy measure near the origin and the resulting density and tail asymptotics, and it presents detailed examples across several standard models. For additional results regarding the tail asymptotics, we also highlight the recent survey by Minchev and Savov \cite{Minchev-savov-2026}.

In this paper, we approximate the density $\phi_{[X]}$ of a general driftless subordinator using families of closed-form densities, and we can attain arbitrarily high accuracy. In full generality,  the density $\phi_{[X]}$ of a driftless subordinator can be expressed through the Mellin--Barnes representation
\begin{equation}\label{2:MB}
 \phi_{[X]}(x)=\frac{1}{2\pi \ii}\int_{\mathrm{Re}\,z=a}x^{-z}\frac{\Gamma(z)}{W_{\Psi}(z)}\,\dd z, \quad x>0,
\end{equation}
where $W_{\Psi}(z)$ is the Bernstein-gamma function and $\Gamma(z)$ denotes the usual gamma function; see, for instance, Minchev and Savov \cite[Section 5]{MS23}. Although this identity is attractive theoretically, it is not well suited for numerical work for several reasons: (i) The contour integral must be recomputed for every $x>0$, (ii) the quotient $\Gamma(z)/W_{\Psi}(z)$ often exhibits a complicated analytic structure (with poles and possibly branch cuts), and (iii) for a wide range of subordinators the function $W_{\Psi}(z)$ is not known in closed form. A different computational approach is to discretize the Volterra equation (\ref{2:k}) over a grid, producing piecewise-constant approximations of the form
\begin{equation}\label{2:ktilde}
  \tilde{k}=\sum_{i=0}^{N-1}\mathbf{1}_{[x_i, x_{i+1})}y_i,
\end{equation}
where $x_i$'s are grid points and the coefficients $y_i$'s are computed iteratively under a normalization constraint. In applications, one typically solves the resulting linear system and then selects the initial value $y_{N-1}$ so that $\int_0^{\infty}\tilde{k}(x)\,\dd x=1$.

By contrast, in this paper we propose a more structured methodology. Specifically, we approximate the L\'evy measure by a sequence of atomic measures supported on a lattice, which leads to explicit formulas for the density, as in (\ref{4.11d}). As $\epsilon \searrow 0$, these approximations converge pointwise to $\phi_{[X]}$, with explicit error bounds for the corresponding cumulative distribution functions.

\subsection{Integer-valued subordinators}\label{S:2:2}
A subordinator $S$ is said to be an IVS when its L\'evy measure is concentrated on the positive integers $\mathbb{Z}_{++}=\{1,2,\dots\}$. Every IVS can be written as a compound Poisson process,
\begin{equation}\label{vt}
  S_t=\sum_{i=1}^{N_t}Z_i,\quad t\geq0,
\end{equation}
where $N$ is a Poisson process with rate $\lambda$, and $\{Z_i\}_{i=1}^{\infty}$ are \text{i.i.d.} nonnegative, integer-valued random variables with mass function $p_k=\PP\{Z_1=k\}$ satisfying $\sum_{k=1}^{\infty}p_k=1$. In this case, the associated L\'evy measure is purely atomic and given by $\nu_S(\dd z)=\lambda \sum_{k=1}^{\infty}p_k\delta_k(\dd z)$. Conversely, any compound Poisson process of the form (\ref{vt}) whose jump sizes are positive integers defines an IVS. Hence, IVSs are exactly the compound Poisson processes with jumps taking values in $\mathbb{Z}_{++}$.

Although defined by a simple support condition, the family of IVSs is quite broad and contains many models relevant in applications. Below we describe several key examples that will be used as test cases throughout the paper.

\begin{example}\label{pp}
Perhaps the most basic (nondeterministic) IVS is the (time-homogeneous) Poisson process $N\equiv(N_{t})_{t\geq0}$ with rate $\lambda>0$. Its L\'evy measure is $\nu=\lambda\updelta_{1}$, where $\updelta_{\cdot}$ denotes the Dirac measure. This is also among the rare cases for which the exponential functional is known to possess a closed-form density. For any $t\geq0$, the law of $N_t$ is $\PP\{N_t=k\}=e^{-\lambda t}(\lambda t)^k/k!$, $k\in\mathbb{Z}_{++}$.
\end{example}

\begin{example}\label{mipp}
A multiple iterated Poisson process (MIPP) can be constructed by successively subordinating independent Poisson processes (Hu \text{et al.} \cite{H26}). Specifically, let $N^1, N^2, \dots, N^n$ be $n$ independent Poisson processes, all with common rate $\lambda>0$, and set $U_t^{(1)}=N_t^1$, $U_t^{(2)}=N_t^2, \dots, U_t^{(n)}=N_t^n$ for $t\ge0$. For integers $1\le k\le m\le n$, define $V_t^{(k,k)}=U_t^{(k)}$ and then iteratively
$V_t^{(k,k+1)}=U^{(k+1)}(V_t^{(k,k)}),\;
V_t^{(k,k+2)}=U^{(k+2)}(V_t^{(k,k+1)}),\dots,\;
V_t^{(k,m)}=U^{(m)}(V_t^{(k,m-1)})$, $t\ge0$.
In this way, $V^{(n)}\equiv V^{(1,n)}$ is obtained by iterating $n$ independent Poisson processes. It is a pure-jump Markov process with jump sizes bounded away from $0$. From these facts one directly obtains
$\PP\{T_1^{(n)}>t\}=e^{-(-\log \PP\{V_1^{(n)}=0\})t}$,
so the first jump time $T_1^{(n)}$ of $V^{(n)}$ is exponentially distributed with parameter $-\log \PP\{V_1^{(n)}=0\}$ for every $n\ge1$.\footnote{Explicit expressions for sojourn times for MIPPs can be found in Hu \text{et al.} \cite{H26}.} Moreover, for any $t\ge0$ and $k\in\mathbb{Z}_{+}$,
$\PP\{V_t^{(n)}=k\mid V_t^{(n-1)}\}=((\lambda V_t^{(n-1)})^k/k!)e^{-\lambda V_t^{(n-1)}}$,
where the (random) cumulative intensity is $\Lambda(t)=\lambda V_t^{(n-1)}$. Hence, the MIPP has a mixed Poisson distribution\footnote{These are Poisson distributions conditional on the rate parameter value ($\Lambda$) and are frequently encountered in actuarial mathematics. For example, the negative-binomial distribution is obtained if $\Lambda$ follows a gamma distribution; see Barndorff-Nielsen \text{et al.} \cite[Section 3.4.1]{BPS12}.} and may be viewed as a repeated iteration of Poisson laws with independent random intensities, each itself Poisson distributed. This feature shows that the MIPP generalizes the Poisson process via repeated subordination and also falls within the class of IVSs. The distribution of $V^{(n)}$ at a fixed time satisfies the recursion
\begin{equation}\label{71}
  \PP\{V_t^{(n)}=k\}=\frac{\lambda^k}{k!}\sum_{j=0}^{\infty}j^ke^{-\lambda j}\PP\{V_t^{(n-1)}=j\},\quad k\in\mathbb{Z}_{++},\;t\geq0.
\end{equation}
By its construction, $V^{(n)}$ exhibits a substantially richer sample-path behavior than a standard Poisson process: its jumps are not restricted to size $1$, and it typically assigns smaller probability to larger jump sizes (Hu \text{et al.} \cite[Proposition 3]{H26}).
\end{example}

\begin{example}\label{nf}
Orsingher and Toaldo \cite{OT15} studied the point process $N^f$ obtained by time-changing a Poisson process $N$ of rate $\lambda$ with a subordinator $H^f$ whose Laplace exponent is the Bernstein function
\[
f(u)=\int_0^{\infty}(1-e^{-uz})\nu(\dd z),\qquad u\ge 0,
\]
where $\nu$ denotes the L\'evy measure of $H^f$ with $\nu(\{0\})=0$. In that work, they proved that the probability mass function at time $t$ solves the following differential-difference equation:
\begin{equation*}
  \frac{\dd \PP\{N^f_{t}=k\}}{\dd t}=-f(\lambda)\PP\{N^f_{t}=k\}+\sum_{l=1}^k\frac{\lambda^l}{l!}\PP\{N^f_{t}=k-l\}\int_0^{\infty}e^{-\lambda z}z^l\nu(\dd z),\;\; k\in\mathbb{Z}_{+},\;t\geq0,
\end{equation*}
adopting the convention that $\sum^{0}_{l=1}\equiv0$.

Let $\nu^f$ be the L\'evy measure of $N^f$, which is purely atomic. One then has $\nu^f(\mathbb{Z}_{++})=\int_0^{\infty}(1-e^{-\lambda z})\nu(\dd z)$ and, for $k\in\mathbb{Z}_{++}\equiv\{1,2,\dots\}$,
\[
\nu^f(\{k\})=\frac{\lambda^k}{k!}\int_0^{\infty}e^{-\lambda z}z^k\nu(\dd z).
\]
If $\{Z_i\}^{\infty}_{i=1}$ is a sequence of \text{i.i.d.} random variables with law $\PP\{Z_1=k\}=\nu^f(\{k\})/\nu^f(\mathbb{Z}_{++})$, $k\in\mathbb{Z}_{++}$, then $N^f$ can be shown to be a compound Poisson process with representation
\[
N_t^f\overset{\rm d}{=}\sum_{i=1}^{N_t^0}Z_i,\quad t\geq0,
\]
where $N^{0}$ is a Poisson process with intensity $\nu^f(\mathbb{Z}_{++})$. Well-known examples from this family are Poisson inverse Gaussian L\'evy processes (Barndorff-Nielsen \text{et al.} \cite{BPS12}) and space-fractional Poisson processes (Orsingher and Toaldo \cite{OT15}).
\end{example}

The examples listed above are far from being a complete collection of cases where our results apply. As pointed out in the remark after Theorem \ref{main}, the conclusions remain valid for any subordinator whose L\'evy measure is supported on discrete nonnegative values, in particular when these values are evenly spaced. Importantly, all such processes exhibit a compound Poisson form with jumps taking values in a discrete set --- an essential feature used throughout our derivations.

\begin{remark}\label{rem2}
Let $X\equiv(X_{t})_{t\geq0}$ be a compound Poisson process of the form $X_t=\sum_{i=1}^{N_t}Z_i$, $t\geq0$, where $N$ is a Poisson process with rate $\lambda>0$ and $\{Z_i\}^{\infty}_{i=1}$ are \text{i.i.d.} nonnegative integer-valued random variables satisfying $\PP\{Z_1=0\}>0$. Then $X$ can be rewritten in distribution as
\begin{equation*}
  X_t\overset{\rm d}{=}\sum_{i=1}^{\tilde{N}_t}\tilde{Z}_i,\quad t\geq0,
\end{equation*}
where $\tilde{N}$ is a Poisson process with intensity $\tilde{\lambda}=\lambda(1-\PP\{Z_1=0\})=\lambda\PP\{Z_1>0\}$, and the variables $\tilde{Z}_i$ are \text{i.i.d.} taking values in the strictly positive integers, with
$\PP\{\tilde{Z}_1=k\}=\PP\{Z_1=k\}/\PP\{Z_1>0\}$ for $k\in\mathbb{Z}_{++}$. This fact will be used in the proof of Theorem \ref{main}.
\end{remark}


\begin{remark}\label{remV}
One can verify that each MIPP $V^{(n)}$ ($n\in\mathbb{Z}_{++}$) introduced in Example \ref{mipp} is itself a compound Poisson process and admits the representation
\begin{equation}\label{mpr}
  V_t^{(n)}\overset{\rm d}{=}\sum_{i=1}^{N_t}Z^{(n-1)}_i,\quad t\geq0,
\end{equation}
where $N$ is a Poisson process with intensity $\lambda$ and $\{Z^{(n-1)}_i\}$ are \text{i.i.d.} random variables distributed as $V_1^{(n-1)}$. Consequently, by Remark \ref{rem2}, $V^{(n)}$ also has the alternative decomposition
\begin{equation}\label{mpr2}
  V_t^{(n)}\overset{\rm d}{=}\sum_{i=1}^{\tilde{N}_t}\tilde{Z}^{(n-1)}_i,\quad t\geq0,
\end{equation}
where $\tilde{N}$ is a Poisson process with intensity $\lambda \PP\{V_1^{(n-1)}\geq1\}$, and $\{\tilde{Z}^{(n-1)}_i\}$ are \text{i.i.d.} with
\[
\PP\{\tilde{Z}^{(n-1)}_1=k\}=\frac{\PP\{V_1^{(n-1)}=k\}}{\PP\{V_1^{(n-1)}\geq1\}},\qquad k\in\mathbb{Z}_{++},
\]
for any fixed $n\geq1$.
\end{remark}

\medskip

\section{Exponential functionals of IVSs}\label{S:3}
In this section we study the densities of the exponential functionals of IVSs, considering both the drifted and driftless cases. For an IVS $S$ with an added drift, the exponential functional is
\begin{equation}\label{4.1}
  I_{q}:=\int^{\infty}_{0}q^{S_{t}+\mu t}\dd t,
\end{equation}
where the parameter $\mu\geq 0$ represents the drift and $q\in(0,1)$ is a base parameter that acts as a scaling factor. Observe that (\ref{4.1}) is simply the standard exponential functional of the rescaled process $(-\log q)(S+\mu\mathrm{id})$, since
$ I_{q}=\int^{\infty}_{0}e^{-(-\log q)(S_{t}+\mu t)}\dd t$.
Using the representation (\ref{vt}), we immediately see that $(-\log q)(S+\mu\mathrm{id})$ is a subordinator with drift $\mu_{q}=(-\log q)\mu\geq0$ and (purely atomic) L\'{e}vy measure
\begin{equation}\label{4.2}
  \nu_{q}(\dd z)=\lambda\sum^{\infty}_{k=1}\PP\{Z_{1}=k\}\updelta_{(-\log q)k}(\dd z),\quad z>0.
\end{equation}

As a result, Theorem 1 of Bertoin and Yor\cite[Theorem 1]{BY05} yields directly that $I_{q}<\infty$ $\PP$-a.s. Furthermore, by Carmona \text{et al.} \cite[Proposition 2.1]{CPY97}, for each $q\in(0,1)$ the distribution of $I_{q}$ is absolutely continuous and has a density $\phi_{q}(x)$, $x>0$, which solves the following Fokker--Planck equation (see also Pardo \text{et al.} \cite{PRVS13} and Vostrikova\cite{V20}):
\begin{equation}\label{4.3}
  (1-\mu_{q}x)\phi_{q}(x)=\int^{\infty}_{x}\nu_{q}\bigg(\bigg(\log\frac{y}{x},\infty\bigg)\bigg)\phi_{q}(y)\dd y,\quad x>0,
\end{equation}
where $\mu_{q}=(-\log q)\mu$ and $\nu_{q}$ is as in (\ref{4.2}).

More generally, obtaining this kind of Fokker--Planck description relies crucially on time-reversal properties of L\'{e}vy processes. In the present framework, this yields, for example, the Markov property of the process $\big(\int^{t}_{0}e^{-(-\log q)(S_{t-}-S_{s-}+\mu(t-s))}\dd t\big)$; passing to the limit $t\to\infty$ then leads to a time-homogeneous characterization.

\subsection{Exponential functionals of IVSs without an embedded drift component}\label{S:3.1}

We next focus on explicit, closed-form representations for the density of the exponential functionals of IVSs. When there is no drift, the equation (\ref{4.3}) becomes explicitly solvable, with $\mu_q=0$ (equivalently, $\mu=0$ in (\ref{4.1})). This provides a computationally convenient expression for the density $\phi_{q}$. The resulting relations belong to the class of functional-differential equations known as pantograph equations, which will play a central role in the sequel. In the driftless case, the Fokker--Planck equation (\ref{2:k}) simplifies to
\begin{equation}\label{pant}
 \phi(x)=\lambda \sum_{k=1}^{\infty}p_k\int_x^{xq^{-k}}\phi(y)\; \dd y, \quad x>0,
\end{equation}
with fixed $q\in (0, 1)$. Differentiating (\ref{pant}) yields the generalized pantograph equation
\begin{equation}\label{pant1}
\phi'(x)-\lambda \sum_{k=1}^{\infty}p_k \big(q^{-k}\phi(q^{-k}x)-\phi(x) \big)=0, \quad x>0.
\end{equation}
Equations such as (\ref{pant1}) are called pantograph equations because they arise, for instance, in models for electric current collection by a pantograph \cite{I93}. The central observation supporting Theorem \ref{main} is that they admit solutions that can be written as generalized Dirichlet series,
\begin{equation}\label{pant2}
  \phi(x)=\sum_{j=0}^{\infty}c_je^{-aq^{-j}x}, \quad x>0,
\end{equation}
for a to-be-determined constant $a>0$ and coefficients $\{c_j\}^{\infty}_{j=0}$. This approach first appeared for Poisson processes in Bertoin \text{et al.} \cite{BBY04} and was subsequently generalized to broader contexts. Here, we prove that the expansion (\ref{pant2}) continues to hold for arbitrary IVSs, and that the coefficients $c_j$'s satisfy a recurrence relation that depends explicitly on the jump probabilities $\{p_k\}^{\infty}_{k=1}$.

The following Theorem presents the principal result of this section.

\begin{theorem}\label{main}
If $\mu=0$, then for every $q\in(0,1)$, we have the series representation
\begin{equation}\label{4.4}
  \phi_{q}(x)=\frac{\lambda\PP\{Z_{1}\geq 1\}}{\sum^{\infty}_{j=0}c_{j}q^j}\sum^{\infty}_{j=0}c_{j}\exp(-\lambda\PP\{Z_{1}\geq 1\}q^{-j}x),\quad x>0,
\end{equation}
where the coefficients $c_{j}$'s are determined by the recurrence relation\footnote{In fact, this recurrence relation admits a compact, closed-form solution, in terms of combinatorial sums (see, e.g., Mallik \cite[Proposition 1]{M98}), which is however omitted here as it is unamenable to implementations.}
\begin{equation}\label{4.5}
  c_{0}=1,\quad c_{j}=\frac{1}{(1-q^{-j})\PP\{Z_{1}\geq
 1\}}\sum^{j}_{k=1}q^{-k}\PP\{Z_{1}=k\}c_{j-k},\quad j\in \mathbb{Z}_{++}.
\end{equation}
\end{theorem}

\begin{proof}
With $\mu_{q}=\mu=0$ and the L\'{e}vy measure in (\ref{4.2}), the Fokker--Planck equation (\ref{4.3}) for the density specializes to the integral equation
\begin{equation}\label{4.6a}
  \phi_{q}(x)=\lambda\sum^{\infty}_{k=1}\PP\{Z_{1}=k\}\int^{xq^{-k}}_{x}\phi_{q}(y)\dd y,\quad x>0.
\end{equation}
Assuming that $\phi_{q}\in\mathcal{C}^{1}(\mathbb{R}_{++})$, differentiating the both sides of (\ref{4.6a}) gives\footnote{The same equation can be obtained by specializing Vostrikova \cite[Equation (33)]{V20}.}
\begin{equation}\label{4.6}
  \frac{\dd}{\dd x}\phi_{q}(x)-\lambda\sum^{\infty}_{k=1}\PP\{Z_{1}=k\}\big(q^{-k}\phi_{q}(q^{-k}x)-\phi_{q}(x)\big)=0,\quad x>0.
\end{equation}

Specifically, Equation (\ref{4.6}) is a generalized pantograph equation of advanced type. Since \emph{usual} pantograph equations of this type are known to admit solutions in the form of a generalized Dirichlet series (Iserles\cite{I93}), after some inspection we adopt the following ansatz:
\begin{equation}\label{4.7}
  \phi_{q}(x)=\sum^{\infty}_{j=0}\tilde{c}_{j}e^{-aq^{-j}x},\quad x>0,
\end{equation}
where $\{\tilde{c}_{j}\}^{\infty}_{j=0}\subset\mathbb{R}$ are unknown coefficients and $a>0$ an unknown constant. This ansatz is automatically in $\mathcal{C}^{\infty}(\mathbb{R}_{++})$ (validating (\ref{4.6})) and also satisfies the necessary right tail limit $\lim_{x\to\infty}\phi_{q}(x)=0$. Assuming that the series in (\ref{4.7}) is absolutely convergent, plugging (\ref{4.7}) into (\ref{4.6}) and simplifying, we have
\begin{align*}
  0&=-a\sum^{\infty}_{j=0}q^{-j}\tilde{c}_{j}e^{-aq^{-j}x}-\lambda\sum^{\infty}_{k=1}\PP\{Z_{1}=k\} \bigg(q^{-k}\sum^{\infty}_{j=0}\tilde{c}_{j}e^{-aq^{-(j+k)}x}-\sum^{\infty}_{j=0}\tilde{c}_{j}e^{-aq^{-j}x}\bigg) \\
  &=\sum^{\infty}_{j=0}e^{-aq^{-j}x}\bigg((\lambda\PP\{Z_{1}\geq 1\}-aq^{-{j}})\tilde{c}_{j} -\lambda\sum^{j}_{k=1}q^{-{k}}\tilde{c}_{j-k}\PP\{Z_{1}=k\}\bigg),
\end{align*}
where the second equality uses that $\sum^{\infty}_{k=1}\PP\{Z_{1}=k\}=\PP\{Z_{1}\geq1\}$. By factoring out the exponentials, it is implied that $\tilde{c}_{j}$'s must satisfy the following recurrence relation:
\begin{equation}\label{4.8}
  (\lambda\PP\{Z_{1}\geq 1\}-aq^{-j})\tilde{c}_{j}-\lambda\sum^{j}_{k=1}q^{-k}\tilde{c}_{j-k}\PP\{Z_{1}=k\}=0,\quad j\in\mathbb{Z}_{+}.
\end{equation}
It is clear that $\tilde{c}_{j}=0$, $\forall j\in\mathbb{Z}_{+}$, if $\tilde{c}_{0}=0$. Thus, let $\tilde{c}_{0}\neq0$. With $j=0$, (\ref{4.8}) gives that $(\lambda\PP\{Z_{1}\geq 1\}-a)\tilde{c}_{0}=0$, which immediately guarantees the choice
\begin{equation}\label{4.9}
  a=\lambda\PP\{Z_1\geq1\}.
\end{equation}

As a candidate density function, the ansatz (\ref{4.7}) needs to be globally integrable and not change sign. While the latter property is obvious from the integral equation (\ref{4.3}), for the integrability it suffices to consider the left tail limit $\lim_{x\searrow0}\phi_{q}(x)=0$ (see Pardo \text{et al.} \cite[Theorem 2.6]{PRVS13}), which is equivalent to the condition $\sum^{\infty}_{j=0}\tilde{c}_{j}=0$. Note that this condition would imply $\lim_{j\to\infty}\tilde{c}_{j}=0$, thus also verifying that the series in (\ref{4.7}) converges absolutely, hence its global integrability.

To see that $\sum^{\infty}_{j=0}\tilde{c}_{j}=0$ indeed, we apply the unilateral Z-transform to both sides of (\ref{4.8}), denoting $\mathcal{Z}(w):=\sum^{\infty}_{j=0}\tilde{c}_{j}w^{-j}$, $\Re w\geq1$, i.e.,
\begin{equation*}
  \sum^{\infty}_{j=0}(\lambda\PP\{Z_{1}\geq 1\}-aq^{-j})\tilde{c}_{j}w^{-j}-\lambda\sum^{\infty}_{j=0}\sum^{j}_{k=1}(qw)^{-k}\tilde{c}_{j-k}w^{-(j-k)}\PP\{Z_{1}=k\}=0,\quad \Re w\geq1.
\end{equation*}
By using the scaling and time-delay properties of the Z-transform as well as (\ref{4.9}), we obtain after rearranging terms
\begin{equation}\label{4.10}
  \lambda\PP\{Z_{1}\geq 1\}(\mathcal{Z}(w)-\mathcal{Z}(qw))-\lambda\mathcal{Z}(w)\sum^{\infty}_{k=1}(qw)^{-k}\PP\{Z_{1}=k\}=0.
\end{equation}
By setting $w=q^{-1}>1$ in (\ref{4.10}), we then have
\begin{equation*}
  \lambda\PP\{Z_{1}\geq 1\}(\mathcal{Z}(q^{-1})-\mathcal{Z}(1))-\lambda\mathcal{Z}(q^{-1})\PP\{Z_{1}\geq1\}=0,
\end{equation*}
implying that $\sum^{\infty}_{j=0}\tilde{c}_{j}=\mathcal{Z}(1)=0$ as desired.\footnote{In fact, an induction argument via (\ref{4.10}) easily shows that $\mathcal{Z}(q^{n})=0$ for every $n\in\mathbb{Z}\cap[1,\bar{n}]$ (for $\bar{n}\in\mathbb{Z}_{++}$) provided that the series $\sum^{\infty}_{k=1}q^{-k\bar{n}}\PP\{Z_{1}=k\}$ converges. Also, note that the existence of $\mathcal{Z}(q^{-1})$ is guaranteed by that of $\mathcal{Z}(1)$.} Thus, the ansatz (\ref{4.7}) is integrable with respect to $x>0$.

The general solution of (\ref{4.8}) is exactly the second equation for $c_{j}$, $j\geq1$, in (\ref{4.5}). To figure out the correct scale of $\tilde{c}_{j}$'s, note that since the density integrates to 1, by using the ansatz (\ref{4.7}) (integrable and sign-definite) and the Fubini theorem, we have that
\begin{equation*}
  \int^{\infty}_{0}\phi_{q}(x)\dd x=\frac{1}{a}\sum^{\infty}_{j=0}\tilde{c}_{j}q^{j}=1.
\end{equation*}
Therefore, with $\tilde{c}_{0}\neq0$, it follows that $\tilde{c}_{j}=ac_{j}/\sum^{\infty}_{i=0}c_{i}q^{i}$, $j\in\mathbb{Z}_{+}$, where $c_{0}=1$ (for simplicity) and $a$ is given by (\ref{4.9}).

This shows that (\ref{4.4}) is a density function that solves (\ref{4.6}). Then, according to Carmona \text{et al.} \cite[Proposition 2.1]{CPY97}, it must be the density of $I_{q}$. The proof is complete.
\end{proof}

By directly integrating the formula (\ref{4.4}), we immediately obtain a series representation for the cumulative distribution function of $I_{q}$.

\begin{corollary}\label{cor:cdf}
If $\mu=0$, then for every $q\in(0,1)$, we have that
\begin{equation}\label{4.cdf}
  \PP\{I_{q}\leq x\}=1-\frac{1}{\sum^{\infty}_{j=0}c_{j}q^{j}}\sum^{\infty}_{j=0}c_{j}q^{j}\exp(-\lambda\PP\{Z_{1}\geq 1\}q^{-j}x),\quad x>0,
\end{equation}
with $c_{j}$'s satisfying the recurrence relation (\ref{4.5}).
\end{corollary}

\begin{proof}
By the absolute convergence of the series in (\ref{4.7}), the Fubini theorem again allows to write
\begin{equation*}
  \PP\{I_{q}\leq x\}=\int^{x}_{0}\phi_{q}(y)\dd y=\frac{\lambda\PP\{Z_{1}\geq 1\}}{\sum^{\infty}_{j=0}c_{j}q^j}\sum^{\infty}_{j=0}c_{j}\frac{1-e^{-\lambda\PP\{Z_{1}\geq 1\}q^{-j}x}}{\lambda\PP\{Z_{1}\geq 1\}q^{-j}},\quad x>0,
\end{equation*}
which simplifies to (\ref{4.cdf}).
\end{proof}

Next we make a few remarks for Theorem \ref{main} and discuss several examples.

\begin{remark}
Based on the compound Poisson form (\ref{vt}), the functional $I_q$ admits the series representation
\begin{equation}\label{3E}
  I_q=\sum_{k=1}^{\infty}\int_{T_{k-1}}^{T_k}q^{S_t}\dd t=\sum_{k=0}^{\infty}q^{\sum_{i=1}^{k}Z_i}E_{k+1},
\end{equation}
where $T_0=0$ and $\{T_k\}^{\infty}_{k=1}$ are the jump times of the Poisson process $N$, and the differences (sojourn times) $\{E_{k}\}^{\infty}_{k=1}$ are a sequence of independent $\text{Exp}(\lambda)$-distributed random variables. From (\ref{3E}) it is clear that
\begin{equation*}
  \E I_q=\frac{1}{\lambda (1-\E q^{Z_1})}.
\end{equation*}
\end{remark}

\smallskip

\begin{remark}\label{ppc}
If $Z_{i}\equiv1$ for all $i\geq1$ in (\ref{vt}), $S=N$ becomes a Poisson process (Example \ref{pp}), and the L\'{e}vy measure (\ref{4.2}) reduces to $\nu_{q}(\dd z)=\lambda\updelta_{-\log q}(\dd z)$, $z>0$. In this case, the recurrence relation (\ref{4.5}) simplifies to
\begin{equation*}
  c_{0}=1,\quad c_{j}=\frac{q^{-1}c_{j-1}}{1-q^{-j}},\quad j\in\mathbb{Z}_{++},
\end{equation*}
which solves for
\begin{equation*}
  c_{j}=\prod^{j}_{i=1}\frac{q^{-1}}{1-q^{-i}}=\frac{(-1)^{j}q^{\binom{j}{2}}}{(q;q)_{j}},
\end{equation*}
where $(\cdot;\cdot)_{\cdot}$ stands for the $q$-Pochhammer symbol. By using the $q$-binomial theorem (see Gasper and Rahman \cite{GR90}), one can show that $\sum^{\infty}_{j=0}c_{j}q^{j}=(q;q)_{\infty}$, which then yields the formula
\begin{equation}\label{4.13}
  \phi_{q}(x)=\frac{\lambda}{(q;q)_{\infty}}\sum^{\infty}_{j=0}\frac{(-1)^{j}q^{\binom{j}{2}}e^{-\lambda q^{-j}x}}{(q;q)_{j}},\quad x>0.
\end{equation}
The explicit formula (\ref{4.13}) coincides with the one presented in Bertoin and Yor\cite[Equation (15)]{BY05}, which was originally discovered by Bertoin \text{et al.} \cite{BBY04} using a probabilistic approach along with techniques from $q$-calculus.

Moreover, Theorem \ref{main} immediately implies a series representation for the Laplace transform of $I_{q}$. Applying the Laplace transform to (\ref{4.4}) termwise, we obtain
\begin{equation*}
  \E e^{-uI_{q}}=\frac{\lambda\PP\{Z_1\geq1\}}{\sum^{\infty}_{j=0}c_{j}q^{j}}\sum^{\infty}_{j=0}\frac{c_{j}}{u+\lambda\PP\{Z_1\geq1\} q^{-j}},\quad\Re u\geq0,
\end{equation*}
where $c_{j}$'s satisfy (\ref{4.5}). When $S=N$, one again has an explicit formula (see Bertoin and Yor\cite[Equation (13)]{BY05}):
\begin{equation*}
  \E e^{-uI_{q}}=\frac{1}{(-u/\lambda;q)_{\infty}},\quad\Re u\geq0.
\end{equation*}
\end{remark}

\smallskip

\begin{example}\label{e:mipp}
Any MIPP $V^{(n)}$ with $n\geq2$ iterations (Example \ref{mipp}) is a compound Poisson process with intensity $\lambda$ and jump size distribution given by (\ref{71}); recall (\ref{mpr}) from Remark \ref{remV}. Then, for any given $n\geq2$, by applying Theorem \ref{main}, the density of the exponential functional $I^{(n)}_q=\int_0^{\infty}q^{V^{(n)}_t}\dd t$, $q\in(0,1)$, is given by
\begin{align*}
  \phi_q^{(n)}(x)=\frac{\lambda\PP\{Z^{(n-1)}_{1}\geq1\}}{\sum_{j=0}^{\infty}c_jq^j}\sum_{j=0}^{\infty}c_j\exp(-\lambda\PP\{Z^{(n-1)}_1\geq1\}q^{-j}x), \quad x>0,
\end{align*}
where $Z^{(n-1)}_{1}\overset{\rm d}{=}V^{(n-1)}_{1}$ (refer to (\ref{71})) and the coefficients are
\begin{align*}
  &c_0=1,\;c_1=\frac{q^{-1}\PP\{Z^{(n-1)}_1=1\}}{(1-q^{-1})\PP\{Z^{(n-1)}_1\geq1\}},\\
  &c_2=\frac{q^{-2}(\PP\{Z^{(n-1)}_1=1\})^2+q^{-2}(1-q^{-1})\PP\{Z_1^{(n-1)}\geq1\} \PP\{Z_1^{(n-1)}=2\}}{(1-q^{-1})(1-q^{-2})(\PP\{Z_1^{(n-1)}\geq1\})^2},\;\dots
\end{align*}
\end{example}

\begin{example}\label{e:sf}
Consider a space-fractional Poisson process, denoted $N^{(\alpha)}$, which is obtained by time-changing a usual Poisson process $N$ by an $\alpha$-stable subordinator with $\alpha\in(0,1)$; see Orsingher and Toaldo \cite[Section 4]{OT15}. The L\'evy measure of the $\alpha$-stable subordinator is $\nu(\dd z)=\alpha z^{-\alpha-1}\dd z/\Gf(1-\alpha)$, for $z>0$. As explained in Example \ref{nf}, we have that $N^{(\alpha)}$ is a compound Poisson process with intensity $\lambda^{\alpha}$ and jump size distribution given by
\begin{equation}\label{psc}
  \PP\{Z_1=k\}=\frac{\alpha}{k!}\frac{\Gf(k-\alpha)}{\Gf(1-\alpha)}=-\frac{(-\alpha)_k}{k!},\quad k\in\mathbb{Z}_{++},
\end{equation}
for which $\PP\{Z_1\geq1\}=1$. Then, Theorem \ref{main} implies that the exponential functional $I_q^{(\alpha)}=\int_0^{\infty}q^{N_t^{(\alpha)}}\dd t$ has a density that reads
\[
  \phi_{q}^{(\alpha)}(x)=\frac{\lambda^{\alpha}}{\sum^{\infty}_{j=0}c_{j}q^{j}}\sum^{\infty}_{j=0}c_{j}e^{-\lambda^{\alpha} q^{-j}x},\quad x>0,
\]
with
\begin{equation*}
  c_{0}=1, \;c_{1}=\frac{\alpha q^{-1}}{1-q^{-1}},\;c_2=\frac{\alpha q^{-2}(\alpha(1+q^{-1})+1-q^{-1})}{2(1-q^{-1})(1-q^{-2})},\;\dots,
\end{equation*}
for any given $q\in(0,1)$ and $\alpha\in(0,1)$.
\end{example}

\begin{example}\label{e:nb}
A negative-binomial process can be viewed as a compound Poisson process with logarithmically distributed jumps; see Xia \cite[Lemma 1]{X20}. Precisely, a L\'{e}vy process $B$ having a marginal negative-binomial distribution with parametrization $r>0$ and $p_{0}\in(0,1)$ can be written as
\begin{equation*}
  B_t=\sum_{i=1}^{N_t}Z_i,\quad t\geq0,
\end{equation*}
where $N$ is a Poisson process with intensity $\lambda=-r\log(1-p_{0})>0$ and $Z_i$'s are \text{i.i.d.} logarithmic random variables with parameter $p_{0}$, namely
\begin{equation*}
  \PP\{Z_1=k\}=\frac{-1}{\log(1-p_{0})}\frac{p_{0}^k}{k},\quad k\in\mathbb{Z}_{++}.
\end{equation*}
An application of Theorem \ref{main} then gives that the exponential functional $I^{(r,p_{0})}_q=\int_0^{\infty}q^{B_t}\dd t$ has a density expressible as
\begin{equation*}
\begin{aligned}
  \phi^{(r,p_{0})}_q(x)=\frac{-r\log(1-p_{0})}{\sum_{j=0}^{\infty}c_jq^j}\sum_{j=0}^{\infty}c_j(1-p_{0})^{rq^{-j}x},\quad x>0,
\end{aligned}
\end{equation*}
with
\begin{align*}
  &c_0=1,\;c_1=-\frac{p_{0}q^{-1}}{(1-q^{-1})\log(1-p_{0})}, \\
  &c_2=\frac{p_{0}^2q^{-2}}{(1-q^{-1})(1-q^{-2})\log^2(1-p_{0})}-\frac{p_{0}^2q^{-2}}{2(1-q^{-2})\log(1-p_{0})},\;\dots,
\end{align*}
for any given $r>0$ and $p_{0},q\in(0,1)$.
\end{example}

\subsection{Exponential functionals of IVSs with additional drift}\label{S:3.2}

Despite being less relevant in practice, the case with a positive drift ($\mu>0$) is examined for completeness. In particular, when the IVS $S$ is a Poisson process, the distribution of $I_{q}$ in (\ref{4.1}) has been explored in Carmona \text{et al.} \cite[Appendix 2]{CPY97}. The key idea is to use the bounded support of $I_{q}$ in this case to break down the associated Fokker--Planck equation into a sequence of nonhomogeneous ordinary differential equations (ODEs), which can be solved stepwise. In the below theorem, we show that this idea can be adapted to general IVSs and provide a compact, semi-closed-form representation for computational ease. First, from (\ref{4.1}), it is clear that if $\mu>0$, then $\PP$-a.s.,
\begin{equation}\label{4.14}
  I_{q}\leq\int^{\infty}_{0}q^{\mu t}\dd t=\frac{1}{(-\log q)\mu},
\end{equation}
and we recall the compound Poisson form (\ref{vt}) for $S$.

\begin{theorem}\label{side}
If $\mu>0$, then for every $q\in(0,1)$, we have the series representation
\begin{equation}\label{4.15}
  \phi_{q}(x)=\frac{1}{C}\sum^{\infty}_{j=0}
  \begin{cases}
    h_{q,j}(x),&\quad\text{if }\displaystyle \frac{q^{j+1}}{\mu_{q}}<x\leq\frac{q^{j}}{\mu_{q}}, \\
    0,&\quad\text{else},
  \end{cases}
  \quad x>0,
\end{equation}
where $\mu_{q}=(-\log q)\mu$ and
\begin{equation}\label{4.16}
  C=\sum^{\infty}_{j=0}\int^{q^{j}/\mu_{q}}_{q^{j+1}/\mu_{q}}h_{q,j}(x)\dd x>0,
\end{equation}
with the basis functions $h_{q,j}$'s satisfying the following functional recurrence relation:
\begin{align}\label{4.17}
  h_{q,0}(x)&=(1-\mu_{q} x)^{\lambda\PP\{Z_{1}\geq1\}/\mu_{q}-1},\quad x\leq\frac{1}{\mu_{q}}, \nonumber\\
  h_{q,j}(x)&=h_{q,0}(x)\bigg((1-q^{j})^{1-\lambda\PP\{Z_{1}\geq1\}/\mu_{q}}h_{q,j-1}\bigg(\frac{q^{j}}{\mu_{q}}\bigg) \nonumber\\
  &\quad-\lambda\sum^{j}_{k=1}q^{-k}\PP\{Z_{1}=k\}\int^{q^{j}/\mu_{q}}_{x}(1-\mu_{q} y)^{-\lambda\PP\{Z_{1}\geq1\}/\mu_{q}}h_{q,j-k}(q^{-k}y)\dd y\bigg), \nonumber\\
  &\quad x\leq\frac{q^{j}}{\mu_{q}},\;j\in\mathbb{Z}_{++}.
\end{align}
\end{theorem}

\begin{proof}
Similar to (\ref{4.6a}), with $\mu_{q}>0$ now, the Fokker--Planck equation (\ref{4.3}) specializes to the integral equation
\begin{equation}\label{4.18}
  (1-\mu_{q}x)\phi_{q}(x)=\lambda\sum^{\infty}_{k=1}\PP\{Z_{1}=k\}\int^{xq^{-k}}_{x}\phi_{q}(y)\dd y,\quad x>0,
\end{equation}
with $\phi_{q}(x)=0$ for $x>1/\mu_{q}$.

Let us define the decreasing sequence
\begin{equation*}
  a_{j}:=\frac{q^{j}}{\mu_{q}},\quad j\in\mathbb{Z}_{+},
\end{equation*}
satisfying that $\bigcup^{\infty}_{j=0}(a_{j+1},a_{j}]=(0,1/\mu_{q}]$. Then, on the assumption that $\phi_{q}\in\mathcal{C}^{1}((a_{j+1},a_{j}))$ and $\phi_{q}$ is left-hand differentiable at $a_{j}$, for all $j\geq0$, the differentiation of (\ref{4.18}) gives
\begin{equation*}
  \frac{\dd}{\dd x}((1-\mu_{q} x)\phi_{q}(x))-\lambda\sum^{\infty}_{k=1}\PP\{Z_{1}=k\}\big(q^{-k}\phi_{q}(q^{-k}x)-\phi_{q}(x)\big)=0,\quad x>0,
\end{equation*}
or after simplification,
\begin{equation}\label{4.19}
  (1-\mu_{q} x)\phi'_{q}(x)+(\lambda\PP\{Z_{1}\geq1\}-\mu_{q})\phi_{q}(x)-\lambda\sum^{\infty}_{k=1}q^{-k}\PP\{Z_{1}=k\}\phi_{q}(q^{-k}x)=0,\quad x>0.
\end{equation}
Here, the derivative $\phi'_{q}(x)$ is understood in the classical sense for each $(a_{j+1},a_{j})\ni x$, $j\geq0$, and as left-hand derivatives whenever $x=a_{j}$.

The functional relations in (\ref{4.17}) can be established with an induction argument. First, by the bound in (\ref{4.14}), $\phi_{q}(x)=0$ for $x>a_{0}$. For $x\in(a_{1},a_{0}]$, $\phi_{q}(q^{-k}x)=0$ for all $k\geq1$, and (\ref{4.19}) yields a first-order homogeneous ODE,
\begin{equation*}
  (1-\mu_{q} x)\phi'_{q}(x)+(\lambda\PP\{Z_{1}\geq1\}-\mu_{q})\phi_{q}(x)=0,\quad x\leq a_{0},
\end{equation*}
which has a positive solution, \text{e.g.} (by setting the constant of integration to 1),
\begin{equation*}
  h_{q,0}(x)=(1-\mu_{q}x)^{\lambda\PP\{Z_{1}\geq1\}/\mu_{q}-1},\quad x\leq a_{0}.
\end{equation*}
Suppose that for arbitrary $j\geq1$, $h_{q,j}$ is the (unique) solution to the nonhomogeneous ODE
\begin{equation}\label{4.20}
  (1-\mu_{q}x)\phi'_{q}(x)+(\lambda\PP\{Z_{1}\geq1\}-\mu_{q})\phi_{q}(x)=\frac{\lambda}{C}\sum^{j}_{k=1}q^{-k}\PP\{Z_{1}=k\}h_{q,j-k}(q^{-k}x),\quad x\leq a_{j},
\end{equation}
subject to the boundary condition $\phi_{q}(a_{j})=h_{q,j-1}(a_{j})/C$, where $C>0$ is some to-be-determined constant, and $\phi_{q}=h_{q,j-k}/C$ on the interval $(a_{j-k+1},a_{j-k}]$, for all $k\leq j$. This implies that $\phi_{q}=h_{q,j}/C$ on $(a_{j+1},a_{j}]$ as well. Thus, for $x\in(a_{j+2},a_{j+1}]$, $\phi_{q}(q^{-k}x)=h_{q,j+1-k}(q^{-k}x)$ for all $k\leq j+1$ while $\phi_{q}(q^{-k}x)=0$ for all $k\geq j+2$ as $q^{-k}x>a_{0}$, based on (\ref{4.19}) we obtain the ODE
\begin{equation*}
  (1-\mu_{q}x)\phi'_{q}(x)+(\lambda\PP\{Z_{1}\geq1\}-\mu_{q})\phi_{q}(x)=\frac{\lambda}{C}\sum^{j+1}_{k=1}q^{-k}\PP\{Z_{1}=k\}h_{q,j+1-k}(q^{-k}x), \quad x\leq a_{j+1},
\end{equation*}
with the boundary condition $\phi_{q}(a_{j+1})=h_{q,j}(a_{j+1})/C$, whose (unique) solution is $h_{q,j+1}$. This shows that given $C>0$, for every $j\geq1$, $\phi_{q}$ is identified as the unique solution to the first-order linear ODE (\ref{4.20}). Solving this ODE subject to the boundary condition gives the second equation in (\ref{4.17}). The fact that $\phi_{q}(x)=Ch_{q,j}(x)$ for $x\in(a_{j+1},a_{j}]$, $j\geq0$, and $\phi_{q}(x)=0$ for $x>a_{0}$ leads to the series in (\ref{4.15}), recalling that $a_{j}=q^{j}/\mu_{q}$.

Given $C>0$, it is clear from (\ref{4.17}) that the functions $h_{q,j}$'s, being continuous, are all uniformly bounded in $j$, which are also nonnegative as a direct consequence from the Fokker--Planck equation (\ref{4.18}). Therefore, $\phi_{q}$ in (\ref{4.15}) is nonnegative and integrable in $x>0$, which, upon normalization with (\ref{4.16}), is a density function, and by further consulting Carmona \text{et al.} \cite[Proposition 2.1]{CPY97} again, it is the density of $I_{q}$ by solving (\ref{4.18}), as desired.
\end{proof}

A few additional remarks are due.

\begin{remark}\label{nd}
The general formula (\ref{4.15}) shows that, despite continuity, it is generally not true that $\phi_{q}\in\mathcal{C}^{1}((0,1/\mu_{q}))$, as only semi-differentiability is guaranteed at each point $q^{j}/\mu_{q}$, $j\in\mathbb{Z}_{++}$.\footnote{This feature is not an oddity but well-known in the theory of delay differential equations with constant delays; see, e.g., Smith \cite{S11}.} Nonetheless, as seen from the functional relations in (\ref{4.17}), one in fact has $\phi_{q}\in\mathcal{C}^{\infty}((q^{j+1}/\mu_{q},q^{j}/\mu_{q}))$, for every $j\in\mathbb{Z}_{++}$, which contrasts with the case with no drift (Theorem \ref{main}), in which $\phi_{q}\in\mathcal{C}^{\infty}(\mathbb{R}_{++})$.
\end{remark}

\begin{remark}
By passing to the limit $\mu\searrow0$, (\ref{4.15}) also unravels the case with no drift. Note that the limiting counterpart of (\ref{4.17}) is given by
\begin{align*}
  h_{q,0}(x)&=e^{-\lambda\PP\{Z_{1}\geq1\}x},\quad x>0, \\
  h_{q,j}(x)&=h_{q,0}(x)\bigg(1-\lambda\sum^{j}_{k=1}q^{-k}\PP\{Z_{1}=k\}\int^{\infty}_{x}e^{\lambda \PP\{Z_{1}\geq1\}y}h_{q,j-k}(q^{-k}y)\dd y\bigg),\quad x>0,\;j\in\mathbb{Z}_{++}.
\end{align*}
Hence, the resulting basis functions are necessarily linear combinations of exponentials of the form $e^{-\lambda \PP\{Z_{1}\geq1\}q^{-j}x}$, $x>0$, $j\in\mathbb{Z}_{+}$. This property agrees with our ansatz choice (\ref{4.7}) for the density in the form of a generalized Dirichlet series.
\end{remark}

\begin{remark}
Straightforward integration of (\ref{4.15}) yields a formula for the cumulative distribution function, namely
\begin{equation*}
  \PP\{I_{q}\leq x\}=1-\frac{1}{C}\sum^{\infty}_{j=0}
  \begin{cases}
    \displaystyle \int^{q^{j}/\mu_{q}}_{x}h_{q,j}(y)\dd y+\sum^{j-1}_{i=0}\int^{q^{i}/\mu_{q}}_{q^{i+1}/\mu_{q}}h_{q,i}(y)\dd y,&\quad\text{if }\displaystyle \frac{q^{j+1}}{\mu_{q}}<x\leq\frac{q^{j}}{\mu_{q}}, \\
    0,&\quad\text{else},
  \end{cases}
\end{equation*}
for $x>0$, where $C$ and $h_{q,j}$'s are as in Theorem \ref{side}.
\end{remark}

The formula (\ref{4.15}), with (\ref{4.16}) and (\ref{4.17}), gives rise to an iterative procedure for computing the density of $I_{q}$ using numerical integration techniques with high precision. Unfortunately, as illustrated through the following simple example, it appears to be an overall extremely arduous task to derive an explicit formula for the density.

\begin{example}\label{e:pdu}
Recall the setting of Remark \ref{ppc}, where $S=N$ is a Poisson process with intensity $\lambda$, and take $\mu=\lambda=1$ for notational conciseness. Consider the standard exponential functional with $q=1/e$ (with subscript suppressed), so that $\mu_{q}=1$. In this case, with the L\'{e}vy measure reduced to $\nu_{1/e}(\dd z)=\updelta_{1}(\dd z)$, the functional relation in (\ref{4.17}) can be significantly simplified, i.e.,
\begin{equation}\label{4.21}
  h_{0}(x)=1>0,\quad h_{j}(x)=h_{j-1}(e^{-j})-e\int^{e^{-j}}_{x}\frac{h_{j-1}(ey)}{1-y}\dd y,
\end{equation}
and after some lengthy calculations, the density of the exponential functional $I=\int^{\infty}_{0}e^{-(N_{t}+t)\dd t}\dd t$ can be expressed as
\begin{equation*}
  \phi(x)=\frac{1}{\sum^{\infty}_{j=0}\int^{e^{-j}}_{e^{-(j+1)}}h_{j}(x)\dd x}\sum^{\infty}_{j=0}
  \begin{cases}
    h_{j}(x),&\quad\text{if } e^{-(j+1)}<x\leq e^{-j}, \\
    0,&\quad\text{else},
  \end{cases}
\end{equation*}
with the following explicit expressions:
\begin{align*}
  h_{0}(x)&=1,\quad x\leq1, \\
  h_{1}(x)&=-e \log (1-x)+1+e (\log (e-1)-1),\quad x\leq\frac{1}{e}, \\
  h_{2}(x)&=e^2 \Li_2\bigg(\frac{e x-1}{e-1}\bigg)-e^2 \Li_2\bigg(-\frac{1}{e}\bigg)-e (1+e (\log (e-1)-1)) \log (1-x) \\
  &\quad+e^2 \log \bigg(\frac{e-e x}{e-1}\bigg) \log (1-e x)+e^2-e+1-e^2 \log (e-1)-e^2 \log (e+1) \log (e-1) \\
  &\quad+e^2 \log (e^2-1) \log (e-1)+e \log (e-1)+e^2 \log (e+1)-e^2 \log (e^2-1),\quad x\leq\frac{1}{e^{2}},\\
  h_{3}(x)&=\cdots;
\end{align*}
here $\Li_{\cdot}(\cdot)$ denotes the polylogarithm, a transcendental function satisfying the recurrence relation
\begin{equation*}
  \Li_{1}(z)=\log(1-z),\quad\Li_{m+1}(z)=\int^{z}_{0}\frac{\Li_{m}(y)}{y}\dd y,\quad m\in\mathbb{Z}_{++},\;|z|<1.
\end{equation*}
Hence, based on the special form of the integral operator in (\ref{4.21}), the basis functions $h_{j}$'s in this case can only contain polylogarithms of integer orders, but the number of terms can grow geometrically --- for instance, even after significant simplification with the aid of computer algebra, the function $h_{3}$ can take more than 20 lines to present.
\end{example}

\begin{remark}
In Example \ref{e:pdu}, if we extend the drift coefficient to $\mu=1/n$, for an integer $n\geq2$, other things unchanged, then from the first equation in (\ref{4.17}), the density $\phi$ will have a polynomial right tail vanishing at $n$, and explicit formulas for $h_{j}$, $j\geq0$, are still available in terms of polylogarithms of integer orders. However, regardless of its complexity, such explicitness is not universal. For instance, by taking $\mu=2$ instead, after some calculations,
\begin{align*}
  h_{0}(x)&=\frac{1}{\sqrt{1-2x}},\quad x\leq\frac{1}{2}, \\
  h_{1}(x)&=\frac{-2 \sqrt{e} \log \big(\sqrt{1-2 e x}+\sqrt{e-2 e x}\big)+\sqrt{e} \log (e-1) +2}{2 \sqrt{1-2 x}},\quad x\leq\frac{1}{2e},
\end{align*}
but the primitive of the function $h_{1}(x)/\sqrt{1-2x}$ does not seem to be known in closed form.
\end{remark}

\smallskip

\begin{example}\label{e:mippd}
Let $S=V^{(n)}$ be an MIPP with $n\geq2$ iterations as in Example \ref{e:mipp}, having jump intensity $\lambda$. Then, for general $\mu>0$, the density of the (standard, with $q=1/e$) exponential functional $I^{(n)}=\int^{\infty}_{0}e^{-(V^{(n)}_{t}+\mu t)}\dd t$ can be written as
\begin{equation*}
  \phi^{(n)}(x)=\frac{1}{\sum^{\infty}_{j=0}\int^{e^{-j}/\mu}_{e^{-(j+1)/\mu}}h^{(n)}_{j}(x)\dd x}\sum^{\infty}_{j=0}
  \begin{cases}
    h^{(n)}_{j}(x),&\quad\text{if }\displaystyle \frac{e^{-(j+1)}}{\mu}<x\leq\frac{e^{-j}}{\mu}, \\
    0,&\quad\text{else},
  \end{cases}
\end{equation*}
where after some simplification,
\begin{align*}
  h^{(n)}_{0}(x)&=(1-\mu x)^{\lambda\PP\{Z^{(n-1)}_{1}\geq1\}/\mu-1},\quad x\leq\frac{1}{\mu}, \\
  h^{(n)}_{1}(x)&=(1-\mu x)^{\lambda\PP\{Z^{(n-1)}_{1}\geq1\}/\mu+1}+\bigg(\frac{e}{e-1}\bigg)^{\lambda\PP\{Z^{(n-1)}_{1}\geq1\}/\mu} \frac{\PP\{Z^{(n-1)}_{1}=1\}}{\PP\{Z^{(n-1)}_{1}\geq1\}} \\
  &\quad\times(1-\mu x)^{\lambda\PP\{Z^{(n-1)}_{1}\geq1\}/\mu-1}(1-e\mu x)^{\lambda\PP\{Z^{(n-1)}_{1}\geq1\}/\mu} \\
  &\quad\times\,_{2}\mathrm{F}_{1}\bigg(\frac{\lambda\PP\{Z^{(n-1)}_{1}\geq1\}}{\mu},\frac{\lambda\PP\{Z^{(n-1)}_{1}\geq1\}}{\mu}; \frac{\lambda\PP\{Z^{(n-1)}_{1}\geq1\}}{\mu}+1\bigg|\frac{e\mu x-1}{e-1}\bigg),\quad x\leq\frac{1}{e\mu}, \\
  h^{(n)}_{2}(x)&=\cdots,
\end{align*}
where recall that $Z^{(n-1)}_{1}$ has the same distribution as $V^{(n-1)}_{1}$, and $\,_{2}\mathrm{F}_{1}(\cdot,\cdot;\cdot|\cdot)$ is the Gauss hypergeometric function, which arises from its familiar integral representation matched to the integral operator in (\ref{4.17}). For the same reason as in the last remark, depending on the exact value of $\mu$, the subsequent function $h^{(n)}_{2}$ (and the ones that follow) may or may not have a closed-form expression.
\end{example}

\section{Exponential functionals of pure-jump subordinators}\label{S:41}

We now shift to the wider --- and substantially harder --- setting of driftless pure-jump subordinators. This category covers many cornerstone models in probability and applications, including gamma, inverse Gaussian, and tempered stable subordinators, whose L\'{e}vy measures live on continuous supports. For such processes, one typically cannot obtain explicit closed-form density formulas comparable to (\ref{4.4}). As aforementioned, current approaches to describing these densities frequently depend on sophisticated analytic tools, such as Mellin--Barnes representations involving Bernstein-gamma functions (Minchev and Savov \cite{MS23} and Patie and Savov \cite{PS18}). Although mathematically powerful, these methods are computationally heavy and do not offer the straightforward applicability of series-based formulas. In what follows we introduce a general, constructive approximation scheme that addresses these issues. The key idea is to exploit the previous explicit density formulas for IVSs to generate a sequence of approximating subordinators. Concretely, given any driftless subordinator $X$, we build IVSs $X^{(\epsilon)}$ by discretizing its L\'evy measure onto a lattice with mesh size $\epsilon$. This produces exponential functionals $I_{[X^{(\epsilon)}]}$ whose densities can be computed explicitly using the generalized Dirichlet series \`{a} la Theorem \ref{main}. We then establish pointwise convergence of these densities to that of $I_{[X]}$ as $\epsilon \searrow 0$, while deriving quantitative convergence bounds for the corresponding cumulative distribution functions. In this regard, our result furnishes a practical and numerically tractable link from the exact IVS setting to the general theory of exponential functionals, thus converting a formerly intractable problem into one that can be approximated systematically for the full class of driftless subordinators.

Recall from Section \ref{S:1} that for a subordinator $X$ the associated exponential functional is
\begin{equation}\label{ix}
I_{[X]}=\int_0^{\infty}e^{-X_t}\dd t.
\end{equation}
As said before, for every subordinator $X$ the random variable $I_{[X]}$ has an infinitely smooth density, which we denote by $\phi_{[X]}$. The purpose of this section is, given an arbitrary driftless subordinator $X$, to construct a sequence of explicit (analytic) densities that converges pointwise to $\phi_{[X]}(x)$. For this purpose, we use the following formulation for a general driftless subordinator:
\begin{equation}\label{4.x}
  X_{t}=\int^{\infty}_{0+}z\,M([0,t],\dd z),\quad t\geq0,
\end{equation}
where $M$ denotes a Poisson random measure on $\mathbb{R}_{+}\times\mathbb{R}_{++}$ with (infinitesimal) intensity $\dd t\times \nu_X(\dd z)$, $\nu_X$ being the L\'evy measure of $X$. This representation is standard in the theory of L\'{e}vy processes (see, e.g., Cont and Tankov \cite[Section 2.6]{CT03}) and expresses $X$ as a superposition of independent Poisson processes, each corresponding to jumps of a given size $z$.
The key insight is that by approximating the jump sizes $z$ in (\ref{4.x}) by a discrete lattice, we obtain a compound Poisson process whose exponential functional density can be computed explicitly. The approximation goes as follows: Fix a discretization parameter $\epsilon \in (0, 1)\cap \mathbb{Q}$, and define the approximation $X^{(\epsilon)}=(X_t^{(\epsilon)})_{t\geq 0}$ by
\begin{equation}\label{el}
X^{(\epsilon)}_{t}:=\int^{\infty}_{\epsilon}\epsilon\bigg\lfloor\frac{z}{\epsilon}\bigg\rfloor \,M([0,t],\dd z),\quad t\geq0.
\end{equation}
In words, each jump of size $z\geq \epsilon$ is replaced by the nearest lower lattice point $\epsilon \lfloor z/\epsilon\rfloor$, while jumps of size $z<\epsilon$ are discarded entirely. The resulting process $X^{(\epsilon)}$ is a compound Poisson process with jump sizes taking values in the discrete set $\{\epsilon k:\; k\in \mathbb{Z}_{++}\}$ and L\'evy measure
\[
\nu_{X^{(\epsilon)}}(\dd z)=\sum_{k=1}^{\infty}\upsilon_{\epsilon, k}\delta_{\epsilon k}(\dd z), \; \; z>0,
\]
where the atoms are given by
\begin{equation}\label{ups}
  \upsilon_{\epsilon,k}:=\nu_{X}([\epsilon k,\epsilon(k+1))), \; \; k\in \mathbb{Z}_{++}.
\end{equation}
The total jump intensity of $X^{(\epsilon)}$ is $\lambda_{\epsilon}:=\sum_{k=1}^{\infty}\upsilon_{\epsilon, k}=\nu_{X}([\epsilon, \infty))<\infty$. The corresponding Laplace exponent is
\[
\Psi^{(\epsilon)}(u)=-\log\E e^{-uX^{(\epsilon)}_{1}}=\sum_{k=1}^{\infty}(1-e^{-u\epsilon k})\upsilon_{\epsilon, k}, \; \; \Re u\geq 0,
\]
and by construction $X^{(\epsilon)}_{t}\overset{\rm d}{\to}X_{t}$ as $\epsilon\searrow0$ for each $t\geq0$, and $\Psi^{(\epsilon)}(u)\rightarrow \Psi(u)$ pointwise for all $\Re u\geq 0$.

For each $\epsilon \in (0, 1)\cap \mathbb{Q}$, define the exponential functional   $I_{[X^{(\epsilon)}]}:=\int^{\infty}_{0}e^{-X^{(\epsilon)}_{t}}\,\dd t$. Since $X^{(\epsilon)}$ is basically an $\epsilon$-scaled IVS, Theorem \ref{main} applies directly. To see this, note that $X^{(\epsilon)}$ can be written as $X_t^{(\epsilon)}=\epsilon \tilde{X}_t^{(\epsilon)},\; \; t\geq 0$, where $\tilde{X}_t^{(\epsilon)}$ is a compound Poisson process with integer-valued jumps:
\[
\tilde{X}_t^{(\epsilon)}=\sum_{i=1}^{N_t^{(\epsilon)}}\tilde{Z}_i^{(\epsilon)}, \; \; t\geq 0,
\]
with $N^{(\epsilon)}$ a Poisson process with intensity $\lambda_{\epsilon}=\nu_X([\epsilon, \infty])$ and jump sizes $\tilde{Z}_i^{(\epsilon)}$'s taking values in $\mathbb{Z}_{++}$ with distribution $\PP\{\tilde{Z}_1^{(\epsilon)}=k\}=\upsilon_{\epsilon, k}/\lambda_{\epsilon}, \;  k\in \mathbb{Z}_{++}$.

Besides, the exponential functional of $X^{(\epsilon)}$ is related to that of $\tilde{X}^{(\epsilon)}$ by a time change. Indeed, setting $q=e^{-\epsilon}$ and recalling the definition (\ref{4.1}) of $I_q$ for an IVS, we have
\[
I_{[X^{(\epsilon)}]}=\int_0^{\infty}e^{-\epsilon \tilde{X}_t^{(\epsilon)}}\dd t=\int_0^{\infty}q^{\tilde{X}^{(\epsilon)}_t}\dd t=I_q^{(\epsilon)},
\]
where $I_q^{(\epsilon)}$ is the exponential functional of the IVS $\tilde{X}^{(\epsilon)}$ with base parameter $q=e^{-\epsilon}$. Applying Theorem \ref{main} to $\tilde{X}^{(\epsilon)}$, we hence obtain the explicit formula
\begin{equation}
\phi_{[X^{(\epsilon)}]}(x)=\phi_q^{(\epsilon)}(x)=\frac{\lambda_{\epsilon}}{\sum_{j=0}^{\infty}c_{\epsilon, j}q^j}\sum_{j=0}^{\infty}c_{\epsilon, j}\exp(-\lambda_{\epsilon}q^{-j}x), \quad x>0,
\end{equation}
where the coefficients satisfy the recurrence relation
\[
c_{\epsilon, 0}=1,\quad c_{\epsilon, j}=\frac{1}{(1-q^{-j})\PP\{\tilde{Z}_1^{(\epsilon)}\geq 1\}}\sum_{k=1}^jq^{-k}\PP\{\tilde{Z}_1^{(\epsilon)}=k\}c_{\epsilon, j-k}, \quad j\in \mathbb{Z}_{++}.
\]
Since $\PP\{\tilde{Z}_1^{(\epsilon)}\geq 1\}=1$ (all jumps are strictly positive), this simplifies to
\[
c_{\epsilon, 0}=1,\quad c_{\epsilon, j}=\frac{1}{1-q^{-j}}\sum_{k=1}^jq^{-k}\frac{\upsilon_{\epsilon, k}}{\lambda_{\epsilon}}c_{\epsilon, j-k}, \quad j\in \mathbb{Z}_{++}.
\]
This representation is the cornerstone of our approximation scheme by expressing the density of $I_{[X^{(\epsilon)}]}$ as a series of elementary exponential functions, with coefficients determined by a simple triangular recurrence.

A crucial ingredient in establishing convergence of the approximating densities is the moment determinacy of the exponential functional $I_{[X]}$. Bertoin and Yor \cite[Theorem 2]{BY05} proved that for any driftless subordinator $X$, the distribution of $I_{[X]}$ is uniquely determined by its moments, which are given explicitly by
\begin{equation}\label{mmt}
  \E I^{m}_{[X]}=\frac{m!}{\prod^{m}_{j=1}\Psi(j)},\quad m\in\mathbb{Z}_{+},
\end{equation}
with the convention $\prod^{0}_{j=1}\equiv1$. Using (\ref{mmt}), it follows immediately that $I_{[X^{(\epsilon)}]}\overset{\rm d}{\to}I_{[X]}$ via the convergence of moments. However, convergence in distribution does not automatically lead to density convergence. The passage from distributional convergence to density convergence requires additional arguments, which we develop next.



To strengthen the convergence beyond distributional convergence, we establish a useful uniform integrability property.

\begin{proposition}\label{pdc}
Let $m\in\mathbb{Z}_{+}$ be fixed. Then the family of random variables $\big\{I_{[X^{(\epsilon)}]}^m:\;\epsilon \in (0, 1]\cap \mathbb{Q}\big\}$ is uniformly integrable. In addition, as $\epsilon\searrow0$ we have $\E I_{[X^{(\epsilon)}]}^m\to \E I_{[X]}^m$.
\end{proposition}

\begin{proof}
It is enough to show that $\sup_{\epsilon\in(0,1)\cap\mathbb{Q}}\E I_{[X^{(\epsilon)}]}^{ml}<\infty$ for some $l>1$. Choose $l=(m+1)/m$, so that $ml=m+1$. By (\ref{mmt}),
\[
\E I_{[X^{(\epsilon)}]}^{m+1}=\frac{(m+1)!}{\prod_{j=1}^{m+1}\Psi^{(\epsilon)}(j)}\to \frac{(m+1)!}{\prod_{j=1}^{m+1} \Psi(j)}
\quad\text{as }\epsilon \searrow 0,
\]
which immediately implies that $\big\{\E I_{[X^{(\epsilon)}]}^{m+1}:\;\epsilon \in (0, 1)\cap \mathbb{Q}\big\}$ is bounded.
\end{proof}

A stronger form of convergence is $\mathrm{L}^1$-convergence of the functionals themselves, which will be used to obtain explicit rates of convergence in distribution.

\begin{theorem}\label{L1}
As $\epsilon\searrow0$, $\E|I_{[X]}-I_{[X^{(\epsilon)}]}|\to0$.
\end{theorem}

\begin{proof}
Subtracting (\ref{el}) from (\ref{4.x}), we obtain
\begin{equation}\label{pab}
  X_{t}-X^{(\epsilon)}_{t}=\int^{\infty}_{\epsilon}\bigg(z-\epsilon\bigg\lfloor\frac{z}{\epsilon}\bigg\rfloor\bigg)M([0,t],\dd z)+\int^{\epsilon}_{0+}zM([0,t],\dd z),\quad t\geq0.
\end{equation}
In particular, (\ref{pab}) implies that $\PP$-a.s. one has $X^{(\epsilon)}_{t}\leq X_{t}$ for all $t\geq0$, and moreover $X^{(\epsilon)}_{t}\nearrow X_{t}$ as $\epsilon\searrow0$. Therefore, by (\ref{ix}),
\begin{align}\label{l1b1}
  \E|I_{[X]}-I_{[X^{(\epsilon)}]}|&\leq\int^{\infty}_{0}\E\big|e^{-X_{t}}-e^{-X^{(\epsilon)}_{t}}\big|\dd t \nonumber\\
  &\leq\int^{\infty}_{0}\E\Big(e^{-X^{(\epsilon)}_{t}}|X_{t}-X^{(\epsilon)}_{t}|\Big)\dd t \nonumber\\
  &\leq\int^{\infty}_{0}\Big(\E e^{-X^{(\epsilon)}_{t}}\E|X_{t}-X^{(\epsilon)}_{t}| +\sqrt{\Var(e^{-X^{(\epsilon)}_{t}})\Var(X_{t}-X^{(\epsilon)}_{t})}\Big)\dd t,
\end{align}
where the second bound follows from the Lipschitz continuity of the exponential map. For (\ref{l1b1}), using the infinite divisibility of $X^{(\epsilon)}_{1}$ together with (\ref{pab}), we further derive
\begin{align}\label{l1b2}
  \mathcal{E}_{1}&:=\int^{\infty}_{0}\E e^{-X^{(\epsilon)}_{t}}\E|X_{t}-X^{(\epsilon)}_{t}|\dd t \nonumber\\
   &=\int^{\infty}_{0}e^{-t\Psi^{(\epsilon)}(1)} \E(X_{t}-X^{(\epsilon)}_{t})\dd t \nonumber\\
  &\leq\int^{\infty}_{0}te^{-t\Psi^{(\epsilon)}(1)}\dd t\bigg(\epsilon\nu_{X}([\epsilon,\infty))+\int^{\epsilon}_{0}z\nu_{X}(\dd z)\bigg) \nonumber\\
  &\leq C_{1}\int^{\epsilon}_{0}\nu_{X}([z,\infty))\dd z \nonumber\\
  &=C_{1}\varrho^{2}(\epsilon),
\end{align}
where
\begin{equation}\label{vrho}
  \varrho(\epsilon):=\sqrt{\int^{\epsilon}_{0}\nu_{X}([z,\infty))\dd z}
\end{equation}
and $C_{1}$ is a constant depending only on $\nu_{X}$. In particular, using the monotonicity of $X^{(\epsilon)}_{1}$ in $\epsilon\in(0,1)\cap\mathbb{Q}$,
\begin{equation*}
  C_{1}=\int^{\infty}_{0}te^{-t\Psi^{(1)}(1)}\dd t=\frac{1}{(\Psi^{(1)}(1))^{2}}>0.
\end{equation*}
The final identity in (\ref{l1b2}) follows from an integration-by-parts formula for Lebesgue--Stieltjes integrals. Analogously,
\begin{align}\label{l1b3}
  \mathcal{E}_{2}&:=\int^{\infty}_{0}\sqrt{\Var(e^{-X^{(\epsilon)}_{t}})\Var(X_{t}-X^{(\epsilon)}_{t})}\dd t \nonumber\\
  &\leq\int^{\infty}_{0}e^{-t\Psi^{(\epsilon)}(2)/2}\sqrt{\Var(X_{t}-X^{(\epsilon)}_{t})}\dd t \nonumber\\
  &\leq\int^{\infty}_{0}\sqrt{t}e^{-t\Psi^{(\epsilon)}(2)/2}\dd t\sqrt{\epsilon^{2}\nu_{X}([\epsilon,\infty))+\int^{\epsilon}_{0}z^{2}\nu_{X}(\dd z)} \nonumber\\
  &\leq\int^{\infty}_{0}\sqrt{t}e^{-t\Psi^{(\epsilon)}(2)/2}\dd t\sqrt{\epsilon\bigg(\epsilon\nu_{X}([\epsilon,\infty))+\int^{\epsilon}_{0}z\nu_{X}(\dd z)\bigg)} \nonumber\\
  &\leq C_{2}\varrho^{2}(\epsilon),
\end{align}
where
\begin{equation*}
  C_{2}=\int^{\infty}_{0}\sqrt{t}e^{-t\Psi^{(1)}(2)/2}\dd t=\frac{\sqrt{2\pi}}{(\Psi^{(1)}(2))^{3/2}}>0.
\end{equation*}

Substituting (\ref{l1b2}) and (\ref{l1b3}) into (\ref{l1b1}) yields
\begin{equation}\label{l1b}
  \E|I_{[X]}-I_{[X^{(\epsilon)}]}|=\mathcal{E}_{1}+\mathcal{E}_{2}\leq(C_{1}+C_{2})\varrho^{2}(\epsilon).
\end{equation}
Finally, since (for subordinators) $\nu_{X}([\epsilon,\infty))=o(1/\epsilon)$ as $\epsilon\searrow0$, we have
\[
\lim_{\epsilon\searrow0}\varrho^{2}(\epsilon)=\lim_{\epsilon\searrow0}\int^{\epsilon}_{0}\nu_{X}([z,\infty))\dd z=0,
\]
which establishes the desired convergence.
\end{proof}

As an immediate consequence of Corollary \ref{cor:cdf}, the cumulative distribution functions of $I_{[X^{(\epsilon)}]}$ for $\epsilon\in(0,1)\cap\mathbb{Q}$ can each be expressed in explicit form. We may then derive the following statement concerning convergence in distribution.

\begin{corollary}\label{cor:main}
Letting $\epsilon\searrow0$, we obtain\footnote{The proof of the corollary contains more precise bounds for the remainder. Moreover, if $\nu_{X}$ is finite (equivalently, if $X$ is a compound Poisson process), then $O(\varrho(\epsilon))=O(\sqrt{\epsilon})$.}
\begin{equation}\label{4.11}
  \PP\{I_{[X]}\leq x\}=1-\frac{1}{\sum^{\infty}_{j=0}c_{\epsilon,j}e^{-\epsilon j}}\sum^{\infty}_{j=0}c_{\epsilon,j}e^{-\epsilon j}\exp\bigg(-e^{\epsilon j}x\sum^{\infty}_{k=1}\upsilon_{\epsilon,k}\bigg)+O(\varrho(\epsilon)),
\end{equation}
uniformly for all $x>0$, where $\upsilon_{\epsilon,k}$ and $\varrho(\epsilon)$ are defined in (\ref{ups}) and (\ref{vrho}). The coefficients $c_{\epsilon,j}$ are specified by
\begin{equation}\label{4.12}
  c_{\epsilon,0}=1,\quad c_{\epsilon,j}=\frac{1}{(1-e^{\epsilon j})\sum^{\infty}_{k=1}\upsilon_{\epsilon,k}}\sum^{j}_{k=1}e^{\epsilon k}c_{\epsilon,j-k}\upsilon_{\epsilon,k},\quad j\in\mathbb{Z}_{++}.
\end{equation}
\end{corollary}

\begin{proof}
From the $\mathrm{L}^{1}$ estimate in (\ref{l1b}) together with Markov's inequality, for any $\eta>0$,
\begin{align*}
  \sup_{x>0}|\PP\{I_{[X]}\leq x\}-\PP\{I_{[X^{(\epsilon)}]}\leq x\}|&\leq\sup_{x>0}\E\big(|\mathbf{1}_{\{I_{[X]}\leq x\}}-\mathbf{1}_{\{I_{[X^{(\epsilon)}]}\leq x\}}| \mathbf{1}_{\{|I_{[X]}-I_{[X^{(\epsilon)}]}|\leq\eta\}}\big) \\
  &\qquad+\PP\{|I_{[X]}-I_{[X^{(\epsilon)}]}|>\eta\} \\
  &\leq\sup_{x>0}\PP\{x-\eta<I_{[X]}\leq x+\eta\}+\eta^{-1}\E|I_{[X]}-I_{[X^{(\epsilon)}]}| \\
  &\leq2\eta\sup_{y>0}\phi_{[X]}(y)+\eta^{-1}(C_{1}+C_{2})\varrho^{2}(\epsilon).
\end{align*}
In addition, $\sup_{y>0}\phi_{[X]}(y)<\infty$ since $\phi_{[X]}$ is continuous and satisfies $\lim_{x\searrow0}\phi_{[X]}(x)=0$ (see Pardo \text{et al.} \cite[Theorem 2.6]{PRVS13}). Choosing $\eta=\varrho(\epsilon)$ yields
\begin{equation}\label{cdf.e}
  |\PP\{I_{[X]}\leq x\}-\PP\{I_{[X^{(\epsilon)}]}\leq x\}|\leq\tilde{C}\varrho(\epsilon),\quad x>0,
\end{equation}
for some constant $\tilde{C}>0$ depending only on $\nu_{X}$.

With (\ref{cdf.e}) in hand, it suffices to observe that the density $\phi_{[X^{(\epsilon)}]}$ for $\epsilon\in(0,1)\cap\mathbb{Q}$ satisfies the Fokker--Planck equation
\begin{equation*}
  \frac{\dd}{\dd x}\phi_{[X^{(\epsilon)}]}(x)-\sum^{\infty}_{k=1}(e^{\epsilon k}\phi_{[X^{(\epsilon)}]}(e^{\epsilon k}x)-\phi_{[X^{(\epsilon)}]}(x))\upsilon_{\epsilon,k}=0,\quad x>0,
\end{equation*}
and, imposing the same boundary conditions, it can be solved exactly as in (\ref{4.6}). This gives
\begin{equation}\label{apdf}
  \phi_{[X^{(\epsilon)}]}(x)=\frac{\sum^{\infty}_{k=1}\upsilon_{\epsilon,k}}{\sum^{\infty}_{j=0}c_{\epsilon,j}e^{-\epsilon j}}\sum^{\infty}_{j=0}c_{\epsilon,j}\exp\bigg(-e^{\epsilon j}x\sum^{\infty}_{k=1}\upsilon_{\epsilon,k}\bigg),\quad x>0,
\end{equation}
where the sequence $\{c_{\epsilon,j}\}$ is determined by (\ref{4.12}) and where $\sum^{\infty}_{k=1}\upsilon_{\epsilon,k}<\infty$. Applying Fubini's theorem to integrate (\ref{apdf}) term-by-term yields
\begin{equation*}
  \PP\{I_{[X^{(\epsilon)}]}\leq x\}=1-\frac{1}{\sum^{\infty}_{j=0}c_{\epsilon,j}e^{-\epsilon j}}\sum^{\infty}_{j=0}c_{\epsilon,j}e^{-\epsilon j}\exp\bigg(-e^{\epsilon j}x\sum^{\infty}_{k=1}\upsilon_{\epsilon,k}\bigg),\quad x>0,
\end{equation*}
which completes the proof.
\end{proof}

We now focus on the central claim of this subsection, i.e., the pointwise convergence of the approximating densities $\phi_{[X^{(\epsilon)}]}$ towards the exact density $\phi_{[X]}$. Note that since all these densities are known to be $C^\infty$ (Carmona \text{et al.} \cite[Proposition 2.1]{CPY97}), it is reasonable to expect that the density $\phi_{[X^{(\epsilon)}]}$ of $I_{[X^{(\epsilon)}]}$ converges pointwise to the density $\phi_{[X]}$ of $I_{[X]}$. Even so, a direct derivation from Corollary \ref{cor:main} seems challenging as we lack sufficient uniform control on bounds for the derivatives of these densities. We therefore adopt an alternative route, relying on moment identities together with suitable estimates for Laplace transforms, which leads to Theorem \ref{main2}.

\begin{theorem}\label{main2}
Let $\epsilon\searrow0$. Then for each $x>0$ we have $\phi_{[X^{(\epsilon)}]}(x)\to\phi_{[X]}(x)$, where
\begin{equation}\label{4.11d}
  \phi_{[X^{(\epsilon)}]}(x)=\frac{\sum^{\infty}_{k=1}\upsilon_{\epsilon,k}}{\sum^{\infty}_{j=0}c_{\epsilon,j}e^{-\epsilon j}}\sum^{\infty}_{j=0}c_{\epsilon,j}\exp\bigg(-e^{\epsilon j}x\sum^{\infty}_{k=1}\upsilon_{\epsilon,k}\bigg),
\end{equation}
and the coefficients $c_{\epsilon,j}$ are specified by (\ref{4.12}).
\end{theorem}

\begin{proof}  The proof proceeds in several steps.

\emph{Step 1.} Fix any $\epsilon\in(0,1)\cap\mathbb{Q}$ and define
\begin{equation}\label{taus}
  \tau_{\epsilon}^\pm:=\int_0^{\infty}(\phi_{[X^{(\epsilon)}]}(x)-\phi_{[X]}(x))^\pm\dd x,
\end{equation}
where $(\cdot)^+$ and $(\cdot)^-$ denote, respectively, the positive and negative parts. Thus, it suffices to show that $\tau^\pm_{\epsilon}\to 0$ as $\epsilon \searrow 0$. To this end, we use the characterization that a real sequence $\{b_n\}^{\infty}_{n=1}$ converges to $b$ if and only if every subsequence admits a further subsequence converging to the same limit $b$. We will apply this observation together with a proof by contradiction.

To begin with, note that the collections $\{\tau_{\epsilon}^+:\epsilon\in(0,1)\cap\mathbb{Q}\}$ and $\{\tau_{\epsilon}^-:\epsilon\in(0,1)\cap\mathbb{Q}\}$ form bounded sequences of nonnegative real numbers. Consequently, every subsequence admits a further subsequence that converges. Therefore, it suffices to show that from any subsequence one can extract a subsubsequence whose limit equals $0$.

Since, for each $m\in\mathbb{Z}_{+}$,
\begin{equation*}
  \E I_{[X]}^m-\E I_{[X^{(\epsilon)}]}^m=\int_0^{\infty}x^m(\phi_{[X^{(\epsilon)}]}(x)-\phi_{[X]}(x))^+\dd x-\int_0^{\infty}x^m(\phi_{[X^{(\epsilon)}]}(x)-\phi_{[X]}(x))^-\dd x,
\end{equation*}
it follows that
\begin{equation*}
  \lim_{\epsilon \searrow 0}\int_0^{\infty}x^m(\phi_{[X^{(\epsilon)}]}(x)-\phi_{[X]}(x))^+\dd x=\lim_{\epsilon \searrow 0}\int_0^{\infty}x^m(\phi_{[X^{(\epsilon)}]}(x)-\phi_{[X]}(x))^-\dd x,
\end{equation*}
which in turn yields the identity
\begin{equation}\label{poly}
  \lim_{\epsilon \searrow 0}\int_0^{\infty}P(x)(\phi_{[X^{(\epsilon)}]}(x)-\phi_{[X]}(x))^+\dd x=\lim_{\epsilon \searrow 0}\int_0^{\infty}P(x)(\phi_{[X^{(\epsilon)}]}(x)-\phi_{[X]}(x))^-\dd x,
\end{equation}
for any polynomial $P(x)$. Setting $P\equiv 1$ in (\ref{poly}) gives
\begin{equation*}
  \lim_{\epsilon \searrow 0}\tau^+_{\epsilon}= \lim_{\epsilon \searrow 0}\tau^-_{\epsilon}.
\end{equation*}
Evidently, the same conclusion remains valid along an arbitrary subsequence $\{\bar{\epsilon}\}$ of $\{\epsilon\}$.\footnote{Here, with some abuse of notation, we denote a subsequence of $\{\epsilon\}$ by itself.}

\emph{Step 2.} Assume for a contradiction that
\begin{equation}\label{aa}
  \lim_{\epsilon \searrow 0}\tau^+_{\epsilon}= \lim_{\epsilon \searrow 0}\tau^-_{\epsilon}>0.
\end{equation}
From (\ref{aa}) we obtain a constant $\gamma>0$ such that $\max\{\tau^+_{\epsilon},\tau^-_{\epsilon}\}\geq \gamma$ for every $\epsilon\in(0,1)\cap\mathbb{Q}$. Introduce the following family of probability measures:
\begin{equation}\label{Qe}
  \Q^+_{\epsilon}(B):=\frac{\int_B(\phi_{[X^{(\epsilon)}]}(x)-\phi_{[X]}(x))^+\dd x}{\tau^+_{\epsilon}}, \quad \Q^-_{\epsilon}(B):=\frac{\int_B(\phi_{[X^{(\epsilon)}]}(x)-\phi_{[X]}(x))^-\dd x}{\tau^-_{\epsilon}},
\end{equation}
where $B$ ranges over Borel subsets of $\R_+$. With these measures in hand, using (\ref{aa}), the relation in (\ref{poly}) is equivalently expressed as
\begin{equation}\label{Qpoly}
  \lim_{\epsilon \searrow 0}\int_0^{\infty}P(x)\Q_{\epsilon}^+(\dd x)=\lim_{\epsilon \searrow 0}\int_0^{\infty}P(x)\Q_{\epsilon}^-(\dd x).
\end{equation}

\emph{Step 3.} We assert that the families $\{\Q^+_{\epsilon}:\epsilon\in(0,1)\cap\mathbb{Q}\}$ and $\{\Q^-_{\epsilon}:\epsilon\in(0,1)\cap\mathbb{Q}\}$ are tight. Indeed, for any $\kappa>1$,
\[
 \Q^+_{\epsilon}([\kappa, \infty))\le \frac{1}{\gamma}\bigg(\int_\kappa^{\infty}\phi_{[X^{(\epsilon)}]}(x)\dd x+\int_\kappa^{\infty}\phi_{[X]}(x)\dd x\bigg).
\]
Since the collection $\{X^{(\epsilon)}:\;\epsilon\in(0,1)\cap\mathbb{Q}\}\cup\{X\}$ is evidently uniformly integrable, the right-hand side of the previous relation can be made arbitrarily small uniformly over all $\epsilon$ and $\kappa$. This yields the tightness of $\{\Q^+_{\epsilon}\}$. By the same reasoning, $\{\Q^-_{\epsilon}\}$ is tight as well. Consequently, there exist probability measures $\Q^+$ and $\Q^-$ on $\mathbb{R}_{+}$ that arise as distributional limits of these two tight families; namely,
\[
 \Q_{\epsilon}^+\overset{\rm d}{\to}\Q^+, \quad \Q_{\epsilon}^-\overset{\rm d}{\to}\Q^-,\quad\text{as }\epsilon\searrow0,
\]
and therefore, for every $u\geq0$,
\begin{equation}\label{limqpm}
  \lim_{\epsilon \searrow 0}\int_0^{\infty}e^{-ux}\Q_{\epsilon}^+(\dd x)=\int_0^{\infty}e^{-ux}\Q^+(\dd x), \quad
  \lim_{\epsilon \searrow 0}\int_0^{\infty}e^{-ux}\Q_{\epsilon}^-(\dd x)=\int_0^{\infty}e^{-ux}\Q^-(\dd x).
\end{equation}

\emph{Step 4.}  We now prove that $\int_0^{\infty}e^{-ux}\Q^+(\dd x)=\int_0^{\infty}e^{-ux}\Q^-(\dd x)$ for all $u\geq 0$. To this end, fix $u>0$. By (\ref{limqpm}), given any sufficiently small $\zeta_{u}>0$ (depending only on $u$), one can choose $\eta_0\equiv\eta_{0}(\zeta_{u})\in(0,1)$ such that
\begin{equation}\label{qepsi}
  \bigg|\int_0^{\infty}e^{-ux}\Q^+(\dd x)-\int_0^{\infty}e^{-ux}\Q^-(\dd x)\bigg|\le 2\zeta_{u}+ \bigg|\int_0^{\infty}e^{-ux}\Q_{\epsilon}^+(\dd x)-\int_0^{\infty}e^{-ux}\Q_{\epsilon}^-(\dd x)\bigg|,
\end{equation}
for any $\epsilon \in (0, \eta_0]\cap \mathbb{Q}$.
By a standard Taylor expansion, we can write
$e^{-ux}=P_{n,u}(x)+((-ux)^n/n!)e^{-\theta_{n,u}(x)}$, where for $n\in\mathbb{Z}_{++}$ and $x>0$,
$P_{n,u}(x)=\sum_{i=0}^{n-1}(-ux)^i/i!$ and $0\le \theta_{n,u}(x)\le ux$. Consequently,
\begin{align}\label{qqep]}
  &\quad\bigg| \int_0^{\infty}e^{-ux}\Q_{\epsilon}^+(\dd x)-\int_0^{\infty}e^{-ux}\Q_{\epsilon}^-(\dd x)\bigg| \nonumber\\
  &\le\bigg |\int_0^{\infty}P_{n,u}(x)\Q_{\epsilon}^+(\dd x)-\int_0^{\infty}P_{n,u}(x)\Q_{\epsilon}^-(\dd x)\bigg| \nonumber\\
  &\quad+\bigg|\int_0^{\infty}\frac{(-ux)^n}{n!}e^{-\theta_{n,u}(x)}\Q_{\epsilon}^+(\dd x)-\int_0^{\infty}\frac{(-ux)^n}{n!}e^{-\theta_{n,u}(x)}\Q_{\epsilon}^-(\dd x)\bigg|.
\end{align}
To control the second absolute difference, applying (\ref{taus}) together with (\ref{Qe}) yields
\begin{align}\label{qqq}
  &\quad\bigg|\int_0^{\infty}\frac{(-ux)^n}{n!}e^{-\theta_{n,u}(x)}\Q_{\epsilon}^+(\dd x)-\int_0^{\infty}\frac{(-ux)^n}{n!}e^{-\theta_{n,u}(x)}\Q_{\epsilon}^-(\dd x)\bigg| \nonumber\\
  &=\frac{u^n}{n!} \bigg|\int_0^{\infty}x^ne^{-\theta_{n,u}(x)}\frac{1}{\tau_{\epsilon}^+}(\phi_{[X^{(\epsilon)}]}(x)-\phi_{[X]}(x))^+\dd x \nonumber\\
  &\quad-\int_0^{\infty}x^ne^{-\theta_{n,u}(x)}\frac{1}{\tau_{\epsilon}^-}(\phi_{[X^{(\epsilon)}]}(x)-\phi_{[X]}(x))^-\dd x\bigg| \nonumber\\
  &\le \frac{u^n}{n!\gamma}\int_0^{\infty}x^n(\phi_{[X^{(\epsilon)}]}(x)-\phi_{[X]}(x))^+\dd x\nonumber\\
  &\le \frac{u^n}{n!\gamma}\big(\E I_{[X^{(\epsilon)}]}^n+\E I_{[X]}^n\big) \nonumber\\
  &=\frac{u^n}{\gamma}\bigg(\frac{1}{\prod^{n}_{j=1}\Psi^{(\epsilon)}(j)}+\frac{1}{\prod^{n}_{j=1}\Psi(j)}\bigg) \nonumber\\
  &\le \frac{2}{\gamma}\frac{u^n}{\prod^{n}_{j=1}\Psi^{(\epsilon)}(j)},\quad u\geq0,
\end{align}
where the penultimate line follows by applying (\ref{mmt}) (recalling that $\Psi^{(\epsilon)}$ and $\Psi$ are the associated Laplace exponents), and the final inequality uses that $\Psi^{(\epsilon)}(u)\nearrow \Psi(u)$ for every $u\ge 0$ as $\epsilon\searrow 0$. This monotone convergence holds since $X^{(\epsilon)}_{1}\nearrow X_{1}$ $\PP$-a.s. Combining (\ref{qepsi}), (\ref{qqep]}), and (\ref{qqq}), we obtain
\begin{align*}
 \bigg|\int_0^{\infty}e^{-ux}\Q^+(\dd x)-\int_0^{\infty}e^{-ux}\Q^-(\dd x) \bigg|&\le2\zeta_{u}+\bigg|\int_0^{\infty}P_{n,u}(x)\Q_{\epsilon}^+(\dd x)-\int_0^{\infty}P_{n,u}(x)\Q_{\epsilon}^-(\dd x)\bigg| \\
 &\quad+\frac{2}{\gamma}\frac{u^n}{\prod^{n}_{j=1}\Psi^{(\epsilon)}(j)},
\end{align*}
for any $\epsilon\in (0,\eta_0]\cap\mathbb{Q}$ and any $n\in\mathbb{Z}_{++}$. Next, fix an arbitrary $\eta_1\in (0,\eta_0)\cap\mathbb{Q}$. Since $\Psi^{(\eta_1)}(n)\nearrow\infty$ as $n\nearrow\infty$, we can choose $n_1$ large enough (depending on $\eta_1$) such that, for all $n\ge n_1$,
\begin{equation}\label{qqp}
 \frac{2}{\gamma}\frac{u^n}{\prod^{n}_{j=1}\Psi^{(\eta_{1})}(j)}\le \zeta_{u}.
\end{equation}
Using again that $\Psi^{(\epsilon)}(u)$ is increasing in $\epsilon$, we see that (\ref{qqp}) continues to hold if $\eta_1$ is replaced by any $\epsilon\in (0,\eta_1]\cap\mathbb{Q}$. Consequently,
\begin{equation*}
  \bigg|\int_0^{\infty}e^{-ux}\Q^+(\dd x)-\int_0^{\infty}e^{-ux}\Q^-(\dd x)\bigg|\le3\zeta_{u}+\bigg|\int_0^{\infty}P_{n_1,u}(x)\Q_{\epsilon}^+(\dd x)-\int_0^{\infty}P_{n_1,u}(x)\Q_{\epsilon}^-(\dd x)\bigg|,
\end{equation*}
for all $\epsilon\in (0,\eta_1]\cap\mathbb{Q}$. Moreover, (\ref{Qpoly}) implies that there exists $\eta_2\le \eta_1$ sufficiently small such that
$\big|\int_0^{\infty}P_{n_1,u}(x)\Q_{\epsilon}^+(\dd x)-\int_0^{\infty}P_{n_1,u}(x)\Q_{\epsilon}^-(\dd x)\big|\le \zeta_{u}$ for all $\epsilon\le \eta_2$, and hence
$\big|\int_0^{\infty}e^{-ux}\Q^+(\dd x)-\int_0^{\infty}e^{-ux}\Q^-(\dd x)\big|\le 4\zeta_{u}$.
Since $\zeta_{u}$ can be taken arbitrarily small, it follows that, for each fixed $u\ge 0$,
\[
\int_0^{\infty}e^{-ux}\Q^+(\dd x)=\int_0^{\infty}e^{-ux}\Q^-(\dd x).
\]
Thus, the Laplace transforms of the probability measures $\Q^+$ and $\Q^-$ agree, and we conclude that $\Q^+=\Q^-$ a.e.

\emph{Step 5.} Now set $\mathbb{A}_{\epsilon}^+:=\{x\in [0, \infty): \phi_{[X^{(\epsilon)}]}(x)\geq \phi_{[X]}(x) \}$ and $\mathbb{A}_{\epsilon}^-:=\{x\in [0, \infty): \phi_{[X^{(\epsilon)}]}(x)<\phi_{[X]}(x) \}$. It is clear that $\mathbb{A}_{\epsilon}^+$ and $\mathbb{A}_{\epsilon}^-$ do not overlap and that, for every $\epsilon \in (0,1]\cap\mathbb{Q}$, their union is $\R_+$. Moreover, note that $\mathrm{supp}\,\Q_{\epsilon}^+ \subset \mathbb{A}_{\epsilon}^+$ and $\mathrm{supp}\,\Q_{\epsilon}^{-} \subset \mathbb{A}_{\epsilon}^{-}$ (understood as almost sure inclusion with respect to Lebesgue measure).
Consequently, one readily obtains $\mathrm{supp}\,\Q^+\subset \liminf \mathbb{A}_{\epsilon}^+$ and $\mathrm{supp}\,\Q^{-}\subset \liminf \mathbb{A}_{\epsilon}^-$. Since the sets $\liminf \mathbb{A}_{\epsilon}^+$ and $\liminf \mathbb{A}_{\epsilon}^-$ are disjoint, it follows that $\mathrm{supp}\,\Q^+\cap \mathrm{supp}\,\Q^{-}=\emptyset$. As we have already established that $\Q^+\equiv \Q^-$, we must have $\Q^+\equiv \Q^-=0$. This, however, contradicts (\ref{aa}), because if $\Q^+\equiv \Q^-=0$, then necessarily
\[
 \lim_{\epsilon \searrow 0}\tau^+_{\epsilon}=\lim_{\epsilon \searrow 0}\int_0^{\infty}(\phi_{[X^{(\epsilon)}]}(x)-\phi_{[X]}(x))^+\dd x=0
\]
and
\[
 \lim_{\epsilon \searrow 0}\tau_{\epsilon}^-=\lim_{\epsilon \searrow 0}\int_0^{\infty}(\phi_{[X^{(\epsilon)}]}(x)-\phi_{[X]}(x))^-\dd x=0.
\]

Finally, observe that the expression (\ref{4.11d}) for $\phi_{[X^{(\epsilon)}]}$ follows directly from (\ref{apdf}). This finishes the proof.\footnote{We acknowledge that the calculations in this proof benefited from valuable discussions with Mikl\'os R\'asonyi.}
\end{proof}

\begin{remark} The approximation scheme in Theorem \ref{main2} reveals an interesting structural property of exponential functionals of subordinators. Specifically, for each $\epsilon>0$, the density in (\ref{apdf}) is supported on the entire positive half line and belongs to the class of generalized Dirichlet series:
\[
\phi_{[X^{(\epsilon)}]}(x)=\sum_{j=0}^{\infty}w_{\epsilon, j}\lambda_{\epsilon, j}e^{-\lambda_{\epsilon, j}x}, \; \; x>0,
\]
where $\lambda_{\epsilon, j}=e^{\epsilon j}\sum_{k=1}^{\infty}\upsilon_{\epsilon, k}$ and $w_{\epsilon, j}=c_{\epsilon, j}e^{-\epsilon j}/(\sum_{i=0}^{\infty}c_{\epsilon, i}e^{-\epsilon i})$ are the mixture weights. Since $\lambda_{\epsilon, j}$ grows geometrically in $j$, the series converges rapidly for moderate to large $x$, making it numerically efficient for tail approximation. Furthermore, the recurrence (\ref{4.12}) for the coefficients $c_{\epsilon, j}$ is of triangular form:  $c_{\epsilon, j}$ depends only on $c_{\epsilon, 0}, \cdots, c_{\epsilon, j-1}$.  This allows for straightforward computation without solving linear systems or performing iterative normalization, in contrast to the grid-based approximation (\ref{2:ktilde}). For practical implementation, the infinite series in (\ref{apdf}) can be truncated at $J=O(|\log\epsilon|)$ terms, as the weights $w_{\epsilon, j}$ decay exponentially for $j$ sufficiently large (see the proof of Theorem \ref{main} for a similar argument in the IVS case). This yields an overall computational complexity of $O(|\log\epsilon|^2)$ operations per evaluation point $x$.
\end{remark}

The convergence statement in Theorem \ref{main2} goes beyond simple convergence in distribution. Indeed, as the proof shows, distributional convergence by itself is insufficient to deduce density convergence; the reasoning relies on a subtle combination of moment convergence (through the moment-determinacy theorem of Bertion and Yor \cite{BY05}), uniform integrability, and a contradiction step that uses the fact that the positive and negative parts of the density difference have disjoint supports. The $\mathrm{L}^1$-bound in Theorem \ref{L1}, together with the ensuing $O(\rho(\epsilon))$ rate in Corollary \ref{cor:main}, yields explicit quantitative control, where $\rho(\epsilon)^2=\int_0^{\epsilon}\nu_{X}([z, \infty])dz$. This rate is dictated by how the L\'evy measure behaves near the origin: When the measure is finite (the compound Poisson setting) it scales as $O(\sqrt{\epsilon})$, whereas for infinite-activity processes it deteriorates in accordance with path irregularity, or the so-called Blumenthal--Getoor index. To the best of our knowledge, this provides the first universal approximation method for densities of exponential functionals that comes with error guarantees.

We note that the convergence rate $\rho(\epsilon)$ depends on the L\'evy measure only via the integral $\int_0^{\epsilon}\nu_{X}([z, \infty))\dd z$. For L\'evy processes whose L\'evy measures vary only very slowly near the origin (e.g., certain slowly varying stable-like processes), this rate can be sluggish, and faster discretization methods (potentially adaptive or importance-weighted) may be required to achieve practical accuracy.

\medskip

\section{Tables and graphs}\label{S:4}

In this section, we present plots illustrating the distribution of the exponential functional of an IVS in (\ref{4.1}). We cover both the driftless setting ($\mu=0$), using the proposed expressions (\ref{4.4}) and (\ref{4.5}) from Theorem \ref{main}, and the case with an added drift ($\mu>0$), employing the general semi-closed-form formulas (\ref{4.15}), (\ref{4.16}), and (\ref{4.17}) from Theorem \ref{side}.

For a compact exposition, in the driftless case we restrict attention to three representative IVSs: an MIPP, a space-fractional Poisson process, and a negative-binomial process, corresponding to Example \ref{e:mipp}, Example \ref{e:sf}, and Example \ref{e:nb}, respectively. We examine four values of the base parameter, $q\in\{1/(2e),1/e,3/(2e),2/e\}$, noting that $q=1/e$ yields the standard exponential functional.

When using Theorem \ref{main}, two series in (\ref{4.4}) must be truncated at some level $K\geq1$: the quasi-geometric series in the denominator and the principal Dirichlet series involving $x>0$. Because $q\in(0,1)$ and the Dirichlet series is asymptotic (in $x$), it suffices to verify that the absolute size of
\begin{equation}\label{num}
  \frac{\lambda\PP\{Z_{1}\geq1\}\sum^{K}_{j=0}c_{j}}{\sum^{K}_{j=0}c_{j}q^{j}} =\frac{\sum^{\infty}_{j=0}c_{j}q^{j}}{\sum^{K}_{j=0}c_{j}q^{j}}\bigg(\mathcal{Z}(1) -\sum^{\infty}_{j=K+1}\tilde{c}_{j}\bigg)
\end{equation}
(recalling the Z-transform) is adequately small. Through (\ref{4.5}) (or (\ref{4.8})), this requirement is governed by how quickly the jump probabilities $\PP\{Z_{1}=k\}$, $k\in\mathbb{Z}_{+}$, decay. Here $\lambda>0$ in (\ref{num}) denotes the jump intensity in the compound Poisson representation (\ref{vt}), and it equals $\lambda$, $\lambda^{\alpha}$, and $-r\log(1-p_{0})$ under the parametrizations of the MIPP, the space-fractional Poisson process, and the negative-binomial process, respectively. In all figures, we enforce the condition $\lambda\PP\{Z_{1}\geq1\}\big|\sum^{K}_{j=0}c_{j}\big/\sum^{K}_{j=0}c_{j}q^{j}\big|<10^{-3}$ in (\ref{num}). Table \ref{tb:1} lists, for each model and the given parameter choices, the smallest truncation level $K$ that still visually meets this requirement for all $q\leq2/e$.\footnote{For the MIPP $V^{(2)}$, $\PP\{Z_{1}\geq1\}=1-e^{-1}$ (with $\lambda=1$), whereas in the other two examples $\PP\{Z_{1}\geq1\}=1$.} The space-fractional Poisson case necessitates a much larger $K$ because its mass function (\ref{psc}) decays slowly;\footnote{This large truncation level is mainly needed to accurately capture the vanishing of the left tail as $x\searrow0$; for $x>0.1$, a considerably smaller truncation (e.g., $K\leq20$) suffices due to the asymptotic character of the Dirichlet series.} indeed, in this case $\E Z_{1}=\infty$. The corresponding density plots are shown in Figure \ref{density_phi}, Figure \ref{density_phi_2}, and Figure \ref{density_phi_3}, respectively.

\begin{table}[H]\small
  \centering
  \caption{Specification for exponential functionals (no drift)}
  \label{tb:1}
  \begin{tabular}{c|c|c}
    \hline
    IVS ($S$) & parameter choice & truncation threshold ($K$) \\ \hline
    MIPP ($V^{(n)}$) & $\lambda=1$, $n=2$ & 8 \\
    space-fractional Poisson process ($N^{(\alpha)}$) & $\lambda=1$, $\alpha=9/10$ & 177 \\
    negative-binomial process ($B$) & $r=2$, $p_{0}=1/2$ & 8 \\
    \hline
  \end{tabular}
\end{table}

\begin{figure}[H]
  \centering
  \includegraphics[scale=0.33]{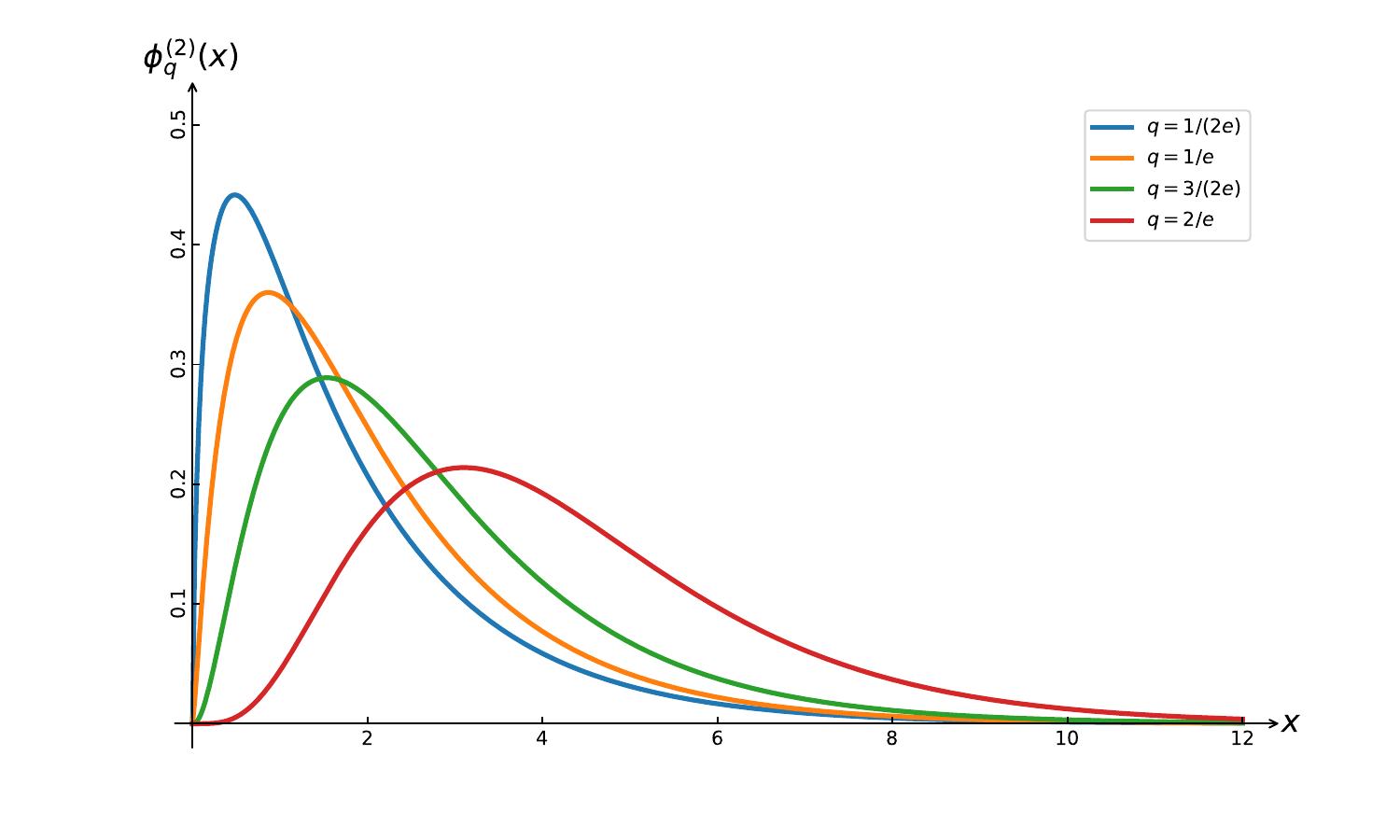}\\
  \caption{Density function of $I_q^{(2)}$ (MIPP case)}
  \label{density_phi}
\end{figure}

\begin{figure}[H]
  \centering
  \includegraphics[scale=0.33]{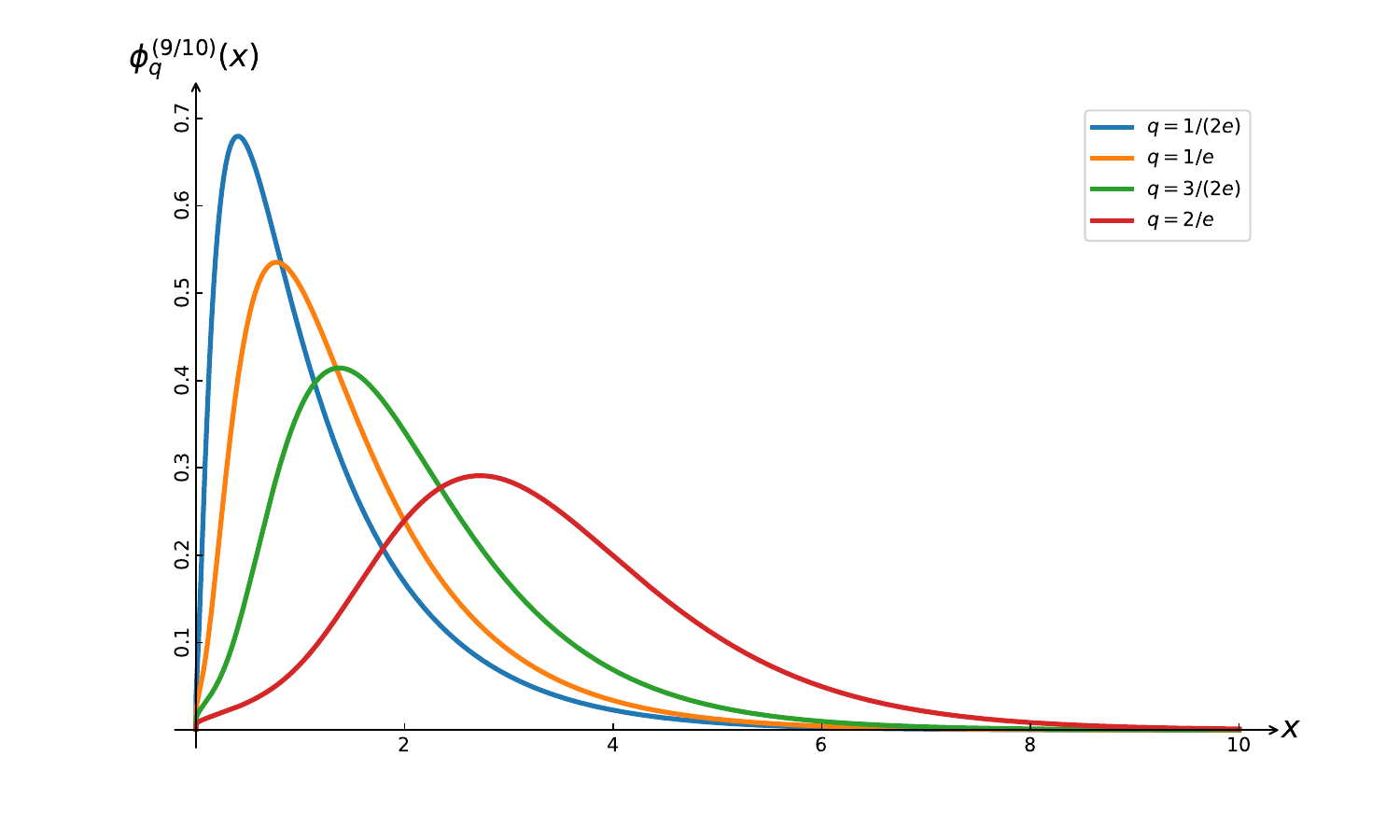}\\
  \caption{Density function of $I_q^{(9/10)}$ (space-fractional Poisson process case)}
  \label{density_phi_2}
\end{figure}

\begin{figure}[H]
  \centering
  \includegraphics[scale=0.33]{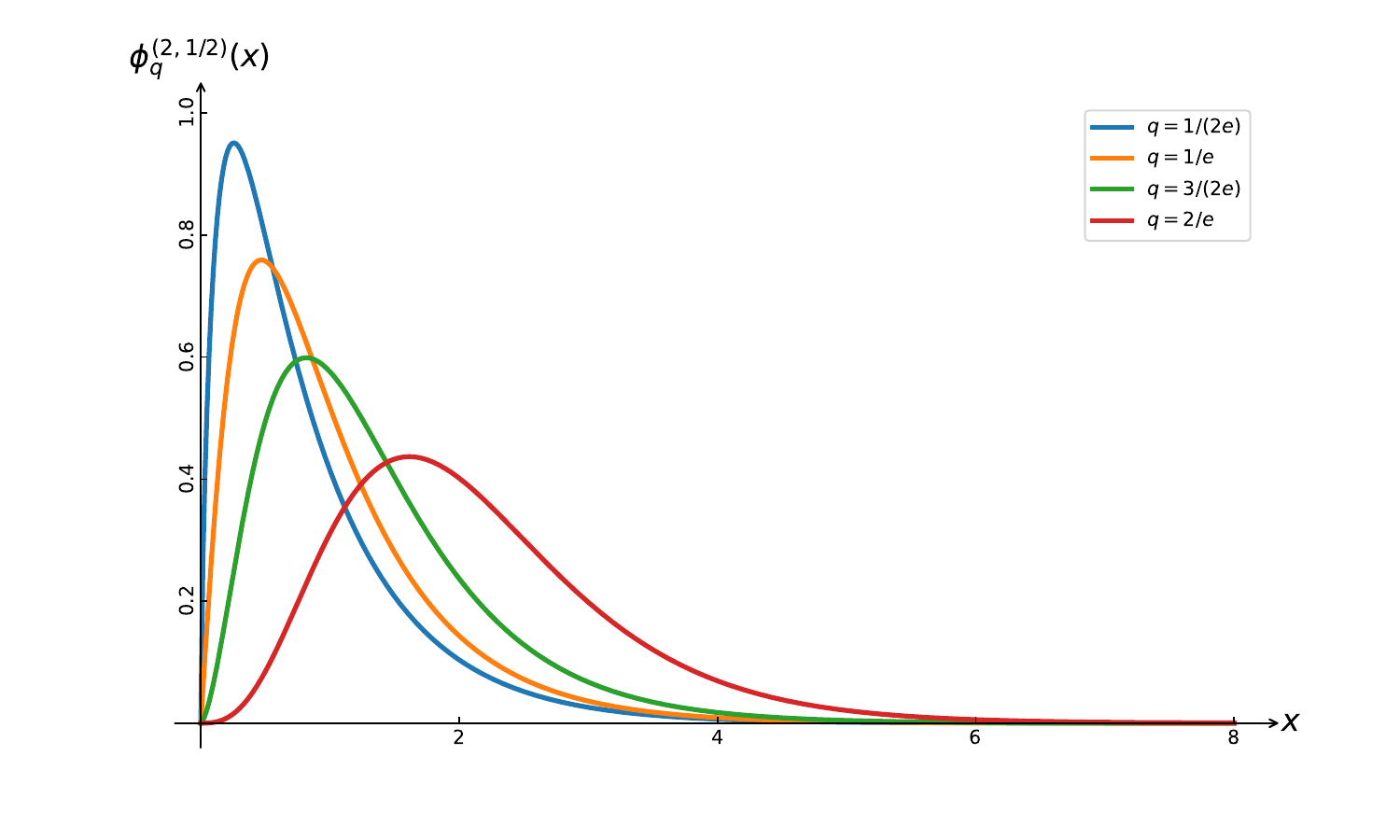}\\
  \caption{Density function of $I_q^{(2,1/2)}$ (negative-binomial process case)}
  \label{density_phi_3}
\end{figure}

When the drift is positive, the analysis becomes substantially more involved. For this reason, we focus on the standard exponential functional with $q=1/e$ and consider only two settings: the case where the IVS is a Poisson process and the case where it is an MIPP, as in Example \ref{e:pdu} and Example \ref{e:mippd}. We examine four choices of the drift parameter, namely $\mu\in\{1/3,1/2,1,2\}$.

To use Theorem \ref{side}, and given the piecewise structure of (\ref{4.15}), it is natural to compute the basis functions $h_{j}\equiv h_{1/e,j}$ only up to some cutoff $K\geq1$. This cutoff is selected so that the interval $(0,e^{-K}/\mu]$ carries only a negligible portion of the total mass, according to the criterion
\begin{equation}\label{dc}
  \frac{\int^{e^{-K}/\mu}_{e^{-(K+1)}/\mu}h_{j}(x)\dd x}{\sum^{K}_{j=0}\int^{e^{-j}/\mu}_{e^{-(j+1)}/\mu}h_{j}(x)\dd x}<10^{-3},
\end{equation}
bearing in mind that the density is zero outside $(0,1/\mu]$; see also (\ref{4.14}). The functions $h_{j}$ for $j\leq K$ are computed exactly up to $j=3$ using the expressions obtained in Example \ref{e:pdu} and Example \ref{e:mippd} (including those in apostrophes). For indices $3<j\leq K$, we then evaluate $h_{j}$ numerically via iteration, using the functional recursion in (\ref{4.17}),\footnote{Observe that the recursion involves only proper integrals of positive, continuous, and nonsingular functions, so numerical computation is straightforward; for instance, using the integration routines in Python's \texttt{scipy.integrate} module.} which also makes it easy to check condition (\ref{dc}). On the remaining interval $(0,e^{-K}/\mu]$, the density value is negligible and is approximated by $0$.

Table \ref{tb:2} lists, for each model, the smallest cutoff $K$ satisfying (\ref{dc}) for all $\mu\geq1/3$, and the corresponding densities appear in Figure \ref{density_phi_4} and Figure \ref{density_phi_5}. One can see that, although the densities are continuous, they are generally only locally (piecewise) differentiable. In particular, in the Poisson setting the right tails are given by line segments or by parabolic or hyperbolic arcs, consistent with the first equation in (\ref{4.17}).

\begin{table}[H]\small
  \centering
  \caption{Specification for exponential functionals (drifted)}
  \label{tb:2}
  \begin{tabular}{c|c|c}
    \hline
    IVS ($S$) & parameter choice & cutoff ($K$) \\ \hline
    Poisson process ($N$) & $\lambda=1$ & 4 \\
    MIPP ($V^{(n)}$) & $\lambda=1$, $n=2$ & 5 \\
    \hline
  \end{tabular}
\end{table}

\begin{figure}[H]
  \centering
  \includegraphics[scale=0.33]{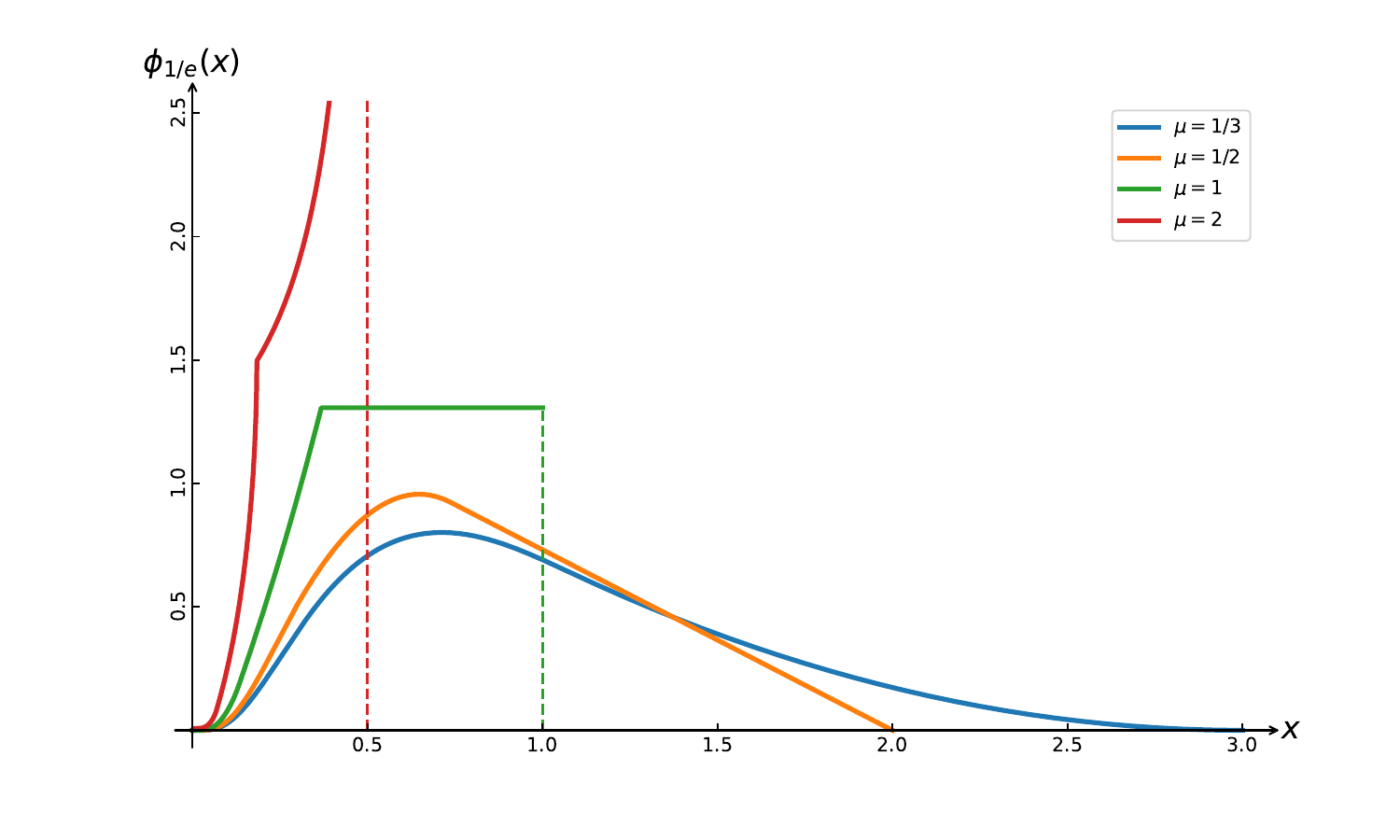}\\
  \caption{Density function of $I\equiv I_{1/e}$ (drifted Poisson case)}
  \label{density_phi_4}
\end{figure}

\begin{figure}[H]
  \centering
  \includegraphics[scale=0.33]{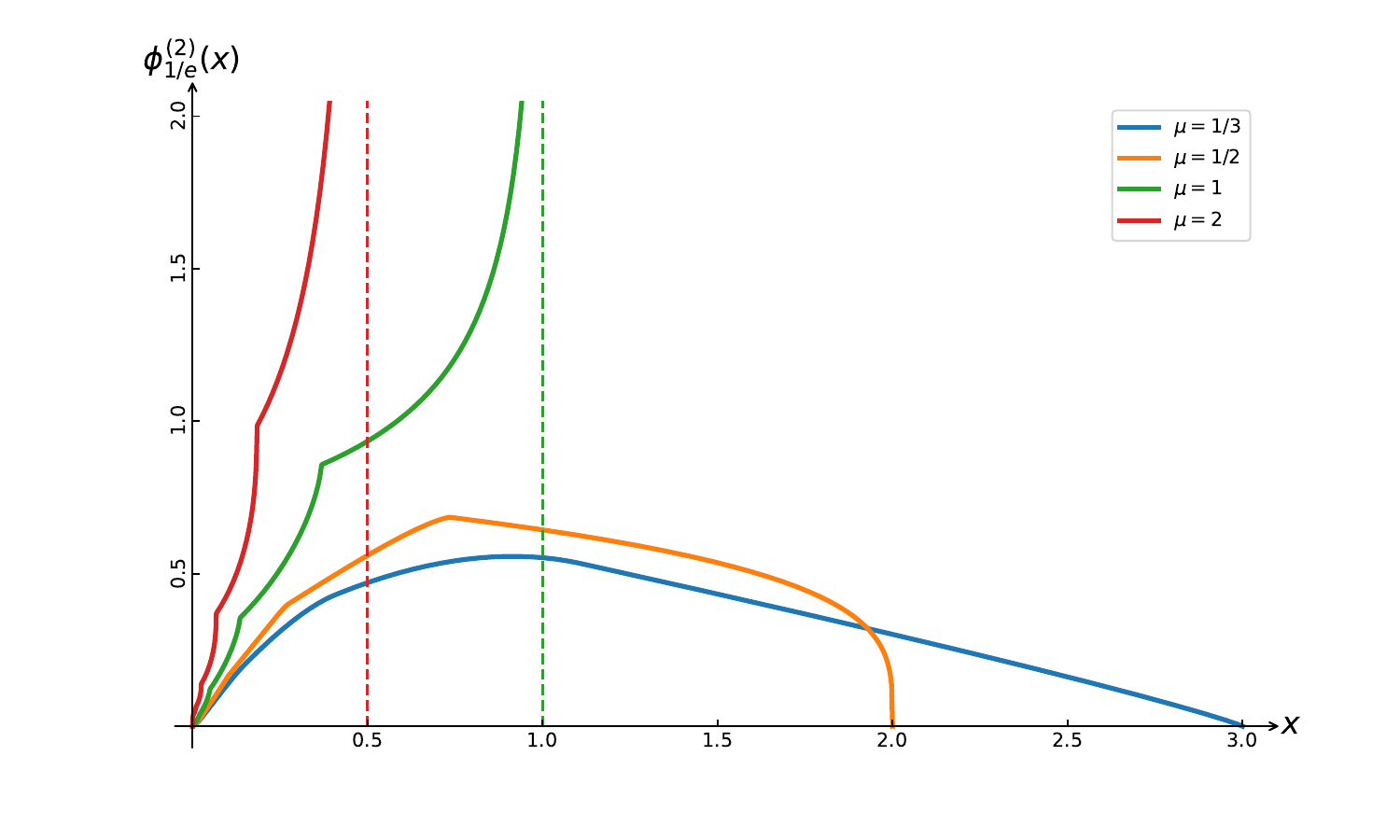}\\
  \caption{Density function of $I^{(2)}\equiv I^{(2)}_{1/e}$ (drifted MIPP case)}
  \label{density_phi_5}
\end{figure}

\noindent \textbf{Specification analysis for Section \ref{S:41}} \medskip

\noindent We now illustrate the general approximation formulas (\ref{4.11}), (\ref{apdf}), and (\ref{4.11d}) in two concrete settings with absolutely continuous L\'{e}vy measures. In the compound Poisson-exponential case, the L\'{e}vy measure is finite and the exponential functional is known to follow a gamma distribution, providing a direct validation of the limit representation (\ref{apdf}). In the tempered stable subordinator case, the L\'evy measure is infinite and the distribution remains largely unknown, yet our approximation yields a constructive limit representation (\ref{4.11d}) involving the incomplete gamma function. These examples demonstrate the versatility of our approach and highlight the analytical challenges that arise in the infinite-activity setting. \medskip

\noindent \underline{Case 1.}\quad Compound Poisson-exponential processes \smallskip\\
This case aims to benchmark against established distributions from the literature, demonstrating the analytical potential of the limit representations. To be more specific, the L\'{e}vy measure is given by $\nu_{X}(\dd z)=ae^{-bz}\dd z$, $z>0$, where $a>0$ and $b>0$ are parameters, and is clearly a finite measure. Equivalently, $X$ is a compound Poisson process with intensity $a/b$ and $\mathrm{Exp}(b)$-distributed jumps.

By applying Corollary \ref{cor:main}, we have that for any $\epsilon\in(0,1)\cap\mathbb{Q}$,
\begin{equation*}
  \upsilon_{\epsilon,k}=\int^{\epsilon(k+1)}_{\epsilon k}ae^{-bz}\dd z=\frac{a}{b}(1-e^{-b\epsilon})e^{-bk\epsilon},\quad\sum^{\infty}_{k=1}\upsilon_{\epsilon,k}=\frac{a}{b}e^{-b\epsilon}.
\end{equation*}
Thus, by Theorem \ref{main2}, for a compound Poisson-exponential process, (\ref{4.11d}) and (\ref{4.12}) specialize to
\begin{equation}\label{new4.1}
  \phi_{[X]}(x)=\lim_{\epsilon\searrow0}\frac{ae^{-b\epsilon}}{b\sum^{\infty}_{j=0}c_{\epsilon,j}e^{-\epsilon j}}\sum^{\infty}_{j=0}c_{\epsilon,j}\exp\bigg(-\frac{a}{b}e^{\epsilon(j-b)}x\bigg), \quad x>0,
\end{equation}
and
\begin{equation}\label{new4.2}
  c_{\epsilon,0}=1,\quad c_{\epsilon,j}=\frac{e^{b\epsilon}-1}{1-e^{\epsilon j}}\sum^{j}_{k=1}c_{\epsilon,j-k}e^{(1-b)k\epsilon},\quad j\in\mathbb{Z}_{++}.
\end{equation}

Here, the limit in (\ref{new4.1}) can be computed explicitly. To briefly explain, note that the limit of the recurrence (\ref{new4.2}) as $\epsilon\searrow0$ is
\begin{equation*}
  \bar{c}_{0}=1,\quad \bar{c}_{j}=-\frac{b}{j}\sum^{j}_{k=1}\bar{c}_{j-k},\quad j\in\mathbb{Z}_{++},
\end{equation*}
with $\bar{c}_{j}:=\lim_{\epsilon\searrow0}c_{\epsilon,j}$, $j\in\mathbb{Z}_{+}$, which via a binomial expansion argument easily solves for
\begin{equation*}
  \bar{c}_{j}=(-1)^{j}\binom{b}{j}.
\end{equation*}
Considering the second series in (\ref{new4.1}) (which is absolutely convergent for fixed $\epsilon$), we have by expanding the exponent in the second series in (\ref{new4.1}) that
\begin{align*}
  \sum^{\infty}_{j=0}c_{\epsilon,j}\exp\bigg(-\frac{a}{b}e^{\epsilon(j-b)}x\bigg) &\sim\sum^{\infty}_{j=0}\bar{c}_{j}\exp\bigg(-\frac{a}{b}(1+\epsilon(j-b))x\bigg) \\
  &=e^{-ax/b+a\epsilon x}(1-e^{-a\epsilon x/b})^{b} \\
  &\sim\bigg(\frac{ax}{b}\bigg)^{b}e^{-ax/b}\epsilon^{b},\quad\text{as }\epsilon\searrow0,
\end{align*}
where ``$\sim$'' denotes asymptotic equivalence and the integration of which (with respect to $x\in\mathbb{R}_{++}$) yields
\begin{equation*}
  \sum^{\infty}_{j=0}c_{\epsilon,j}e^{-\epsilon j}\sim\Gf(b+1)\epsilon^{b},\quad\text{as }\epsilon\searrow0.
\end{equation*}
Plugging these into (\ref{new4.1}) gives that
\begin{equation}\label{new4.3}
  \phi_{[X]}(x)=\lim_{\epsilon\searrow0}\frac{ae^{-b\epsilon}}{b\Gf(b+1)}\bigg(\frac{ax}{b}\bigg)^{b}e^{-ax/b} =\frac{(a/b)^{b+1}}{\Gf(b+1)}x^{b}e^{-ax/b},\quad x>0.
\end{equation}

This case can be recovered from the setting of Pardo \text{et al.} \cite[Example 2]{PRVS13}, with the parameter changes $\beta\mapsto a/b$, $a\mapsto 1$, and $s\mapsto b+1$ in their notation, under which the density formula (\ref{new4.3}) coincides with Pardo \text{et al.} \cite[Equation (4.1)]{PRVS13}. From (\ref{new4.3}) we recognize that $I_{[X]}\overset{\rm d}{=}\mathrm{Gamma}(b+1,a/b)$ (gamma distribution).

\medskip

\noindent \underline{Case 2.}\quad Tempered stable subordinators \smallskip\\
Let $X$ be a tempered stable subordinator, which is characterized by three parameters $a>0$, $b>0$, and $\chi\in[0,1)$. The L\'{e}vy measure is given by $\nu_{X}(\dd z)=(ae^{-bz}/z^{\chi+1})\dd z$, $z>0$. Tempered stable subordinators constitute a fairly general class of infinitely-active subordinators including the gamma process (when $\chi=0$) and inverse Gaussian process (when $\chi=1/2$) as special cases, and have been widely applied to model risky asset prices, especially capturing the price--volume relationship; see, e.g., K\"{u}chler and Tappe \cite[Section 10]{KT13} and Fei and Xia \cite[Section 4.2]{FX24}. In this case, the distribution of $I_{[X]}$ remains largely unknown.\footnote{We mention Minchev and Savov \cite[Section 4.2]{MS23} for the tail behaviors in the case of a gamma subordinator ($\chi=0$).}

By applying Corollary \ref{cor:main} again, we have that for any $\epsilon$,
\begin{equation*}
  \upsilon_{\epsilon,k}=\int^{\epsilon(k+1)}_{\epsilon k}\frac{ae^{-bz}}{z^{\chi+1}}\dd z=ab^{\chi}(\Gf(-\chi,b\epsilon k)-\Gf(-\chi,b\epsilon(k+1))),\quad\sum^{\infty}_{k=1}\upsilon_{\epsilon,k}=ab^{\chi}\Gf(-\chi,b\epsilon),
\end{equation*}
where $\Gf(\cdot,\cdot)$ denotes the upper incomplete gamma function. Specifying (\ref{4.11d}) and (\ref{4.12}) from Theorem \ref{main2} then gives the following limit representation:
\begin{equation}\label{4.411}
  \phi_{[X]}(x)=\lim_{\epsilon\searrow0}\frac{ab^{\chi}\Gf(-\chi,b\epsilon)}{\sum^{\infty}_{j=0}c_{\epsilon,j}e^{-\epsilon j}}\sum^{\infty}_{j=0}c_{\epsilon,j}\exp(-ab^{\chi}e^{\epsilon j}\Gf(-\chi,b\epsilon)x),\quad x>0,
\end{equation}
with
\begin{equation*}
  c_{\epsilon,0}=1,\quad c_{\epsilon,j}=\frac{1}{(1-e^{\epsilon j})\Gf(-\chi,b\epsilon)}\sum^{j}_{k=1}e^{\epsilon k}c_{\epsilon,j-k}(\Gf(-\chi,b\epsilon k)-\Gf(-\chi,b\epsilon(k+1))),\quad j\in\mathbb{Z}_{++}.
\end{equation*}
While a valid limit representation, computing (\ref{4.411}) explicitly appears to be much harder (if possible at all) present the incomplete gamma functions, a particular subtlety being that both infinite series possess a growth (rather than decay) rate in $\epsilon$ because $\nu_{X}$ is now an infinite measure. A rigorous analysis of the limit in (\ref{4.411}), or that in the general case as in Theorem \ref{main2}, therefore points to a meaningful direction for future development.

\smallskip

\begin{figure}[H]
  \centering
  \begin{minipage}[l]{0.49\linewidth}
  \centering
  \includegraphics[scale=0.3]{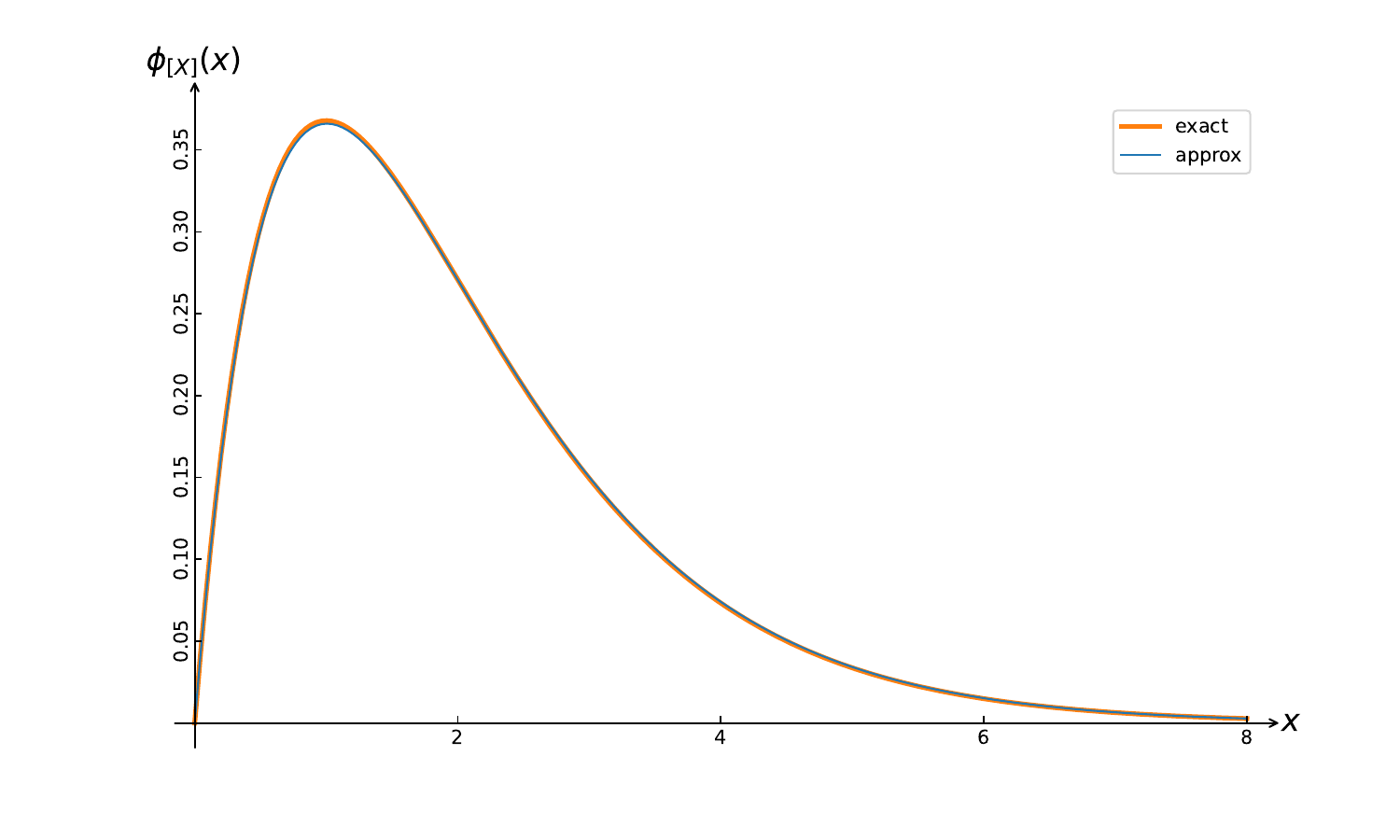}
  \caption*{\small (compound Poisson-exponential process case; \\ $\mathrm{Gamma}(2,1)$ distribution)}
  \end{minipage}
  \begin{minipage}[r]{0.49\linewidth}
  \includegraphics[scale=0.3]{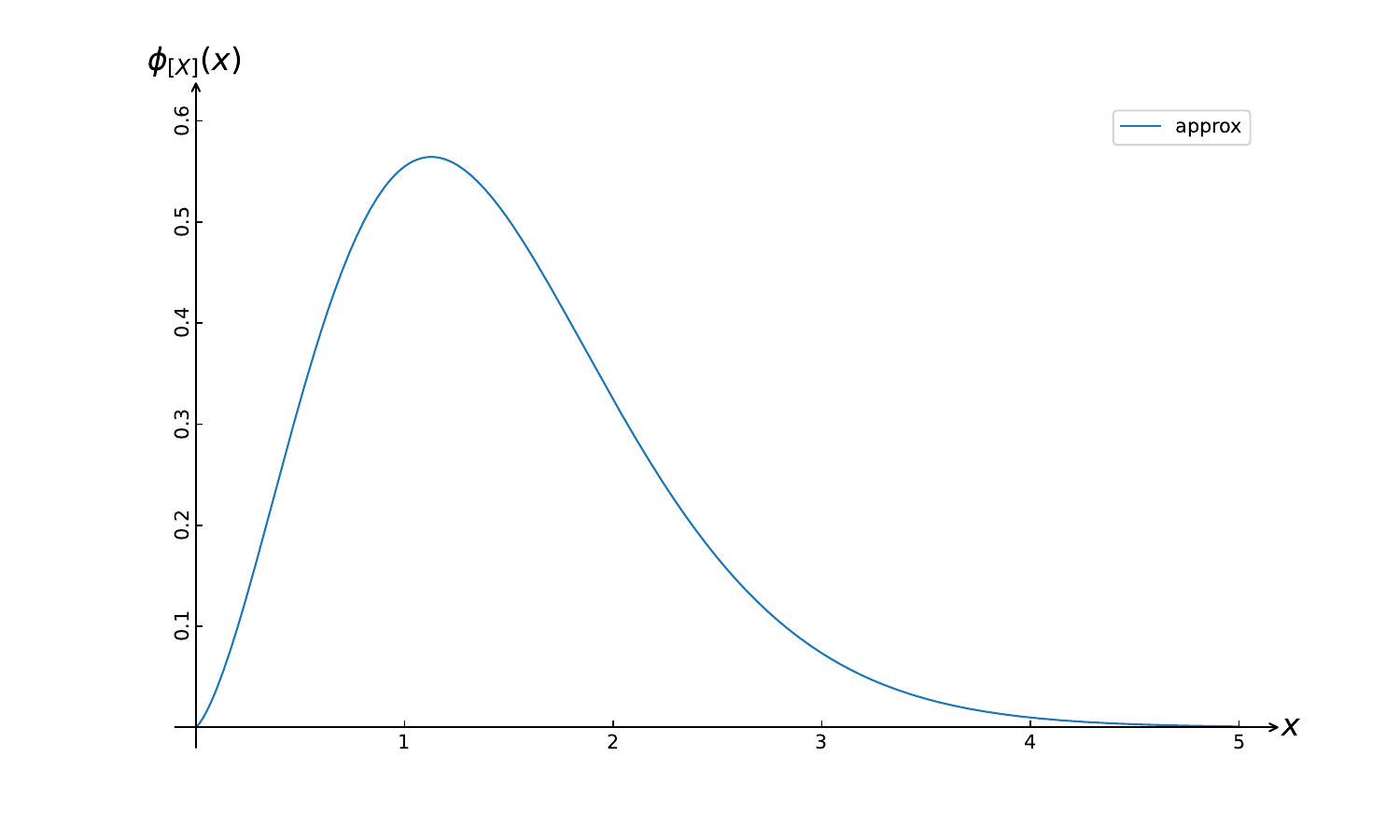}
  \caption*{\small (gamma subordinator case; \\ unnamed distribution)}
  \end{minipage}
  \caption{Density function of $I_{[X]}$}
  \label{density_phi[X]}
\end{figure}

We briefly present the densities for the two cases above to demonstrate the numerical accuracy of the proposed limit representation, and for brevity we omit the computational details. For the compound Poisson-exponential case, we set $a=b=1$ and evaluate formula (\ref{new4.1}) with $\epsilon=0.01$, truncating the infinite series at 100 terms, and compare the result with the closed-form expression in (\ref{new4.3}). As shown in Figure \ref{density_phi[X]}, when no analytic simplification is performed, the limit representation still provides a satisfactory numerical approximation. For the tempered stable subordinator case, we take $\chi=0$ (so the model reduces to a gamma subordinator) and again set $a=b=1$. We then approximate the density via (\ref{4.411}) using $\epsilon=2\times10^{-3}$ and truncating the infinite series at 5000 terms; the outcome is also reported in Figure \ref{density_phi[X]}.

We remark that these two examples are selected to showcase the breadth of the method. In the compound Poisson-exponential setting, the L\'{e}vy measure is finite and the exponential functional is known to have a gamma law; the fact that our limit representation (\ref{new4.3}) reproduces the gamma density exactly and, after a reparametrization, coincides with the expression in Pardo \text{et al.} \cite[Equation (4.1)]{PRVS13}, hence offering a straightforward check of the procedure. In the tempered stable setting, on the other hand, the L\'{e}vy measure is infinite and the distribution of $I_{[X]}$ is still largely unknown; in this case, our method provides a constructive limit representation (\ref{4.411}) in terms of the incomplete gamma function, which shows that the scheme extends in a meaningful way beyond the class of exactly solvable models.

\medskip

\section{Conclusion}\label{S:6}

This paper presents a global approach for the densities of exponential functionals of subordinators, encompassing both fully explicit, exactly solvable settings and general approximation methods equipped with rigorous convergence assurances. For IVSs, Theorem \ref{main} yields an explicit form in terms of an absolutely convergent generalized Dirichlet series, whose coefficients are obtained through a straightforward triangular recurrence relation. This single result subsumes all earlier explicit formulas, including those for Poisson, space-fractional Poisson, multiple iterated Poisson, and negative-binomial processes --- within one analytic approach based on pantograph equations and $q$-calculus. In the presence of drift, Theorem \ref{side} provides a constructive iterative scheme that expresses the density using basis functions solving first-order linear ODEs over successive intervals. For arbitrary driftless subordinators with continuous L\'{e}vy measures, Section \ref{S:41} proposes a general-purpose approximation procedure: It discretizes the L\'{e}vy measure and then applies the exact IVS expressions. Theorem \ref{main2} proves pointwise convergence of the resulting approximating densities, and Corollary \ref{cor:main} supplies explicit convergence rates. To the best of our knowledge, this is the first universal approximation method for exponential functional densities that combines explicit formulas with provable error bounds, effectively reducing computation to summing elementary exponentials via triangular recurrences --- an approach that is numerically more stable than Mellin--Barnes representations and Fourier inversion.

These results are important from both practical and theoretical perspectives. The IVS formulation expresses the density as a linear combination of elementary exponential components, with coefficients obtained via an $O(|\log\epsilon|^2)$ recurrence, which provides a stable, transparent procedure that avoids Mellin--Barnes contour integration and reliance on special functions. The associated approximation scheme extends these benefits to general driftless subordinators, replacing numerically delicate inversion with a principled and systematic construction. Promising avenues for future work include generalizing the method to drifted subordinators with continuous L\'evy measures, treating wider families of decreasing (including inverse-power) functionals, and analyzing exponential functionals of two-sided L\'evy processes within the same pantograph-$q$-calculus-Poisson framework.


\bigskip


\bigskip

\begin{thebibliography}{99}\footnotesize
\bibitem{BPS12} Barndorff-Nielsen, O.E., Pollard, D.G., \& Shephard, N. (2012). Integer-valued L\'evy processes and low latency financial econometrics. \textsl{Quantitative Finance}, 12(4): 587--605.
\bibitem{BBY04} Bertoin, J., Biane, P., \& Yor, M. (2004). Poissonian exponential functionals, $q$-series, $q$-integrals, and the moment problem for log-normal distributions. In: \textsl{Seminar on Stochastic Analysis, Random Fields and Applications IV}. Centro Stefano Franscini, Ascona, May 2002 (pp. 45--56). Birkh\"{a}user Basel.
\bibitem{BLM08} Bertoin, J., Lindner, A., \& Maller, R. (2008). On continuity properties of the law of integrals of L\'evy processes. In: \textsl{S\'eminaire de Probabilit\'es XLI}. Lecture Notes in Mathematics, 1934 (pp. 137--159). Springer Berlin Heidelberg.
\bibitem{BY05} Bertoin, J., \& Yor, M. (2005). Exponential functionals of L\'{e}vy processes. \textsl{Probability Surveys}, 2: 191--212.
\bibitem{CK12} Cai, N., \& Kou, S.G. (2012). Pricing Asian options under a hyper-exponential jump diffusion model. \textsl{Operations Research}, 60(1): 64--77.
\bibitem{CPY97} Carmona, P., Petit, F., \& Yor, M. On the distribution and asymptotic results for exponential functionals of L\'{e}vy processes, In: \textsl{Exponential functionals and principal values related to Brownian motion}, 1997 (pp. 73--130). Biblioteca de la Revista Matematica IberoAmericana.
\bibitem{CKP09} Chaumont, L., Kyprianou, A.E., \& Pardo, J.C. (2009). Some explicit identities associated with positive self-similar Markov processes. \textsl{Stochastic Processes and Their Applications}, 119(3): 980--1000.
\bibitem{CT03} Cont, R., \& Tankov, P. \textsl{Financial Modelling with Jump Processes}, Chapman and Hall / CRC Press, 2003.
\bibitem{EM05} Erickson, K.B., \& Maller, R.A. (2005). Generalised Ornstein--Uhlenbeck processes and the convergence of L\'evy integrals. \textsl{S\'eminaire de Probabilit\'es XXXVIII}, 70--94.
\bibitem{FX24} Fei, Z., \& Xia, W. (2024). Regulating stochastic clocks. \textsl{Quantitative Finance}, 24(7): 921--953.
\bibitem{GR90} Gasper, G., \& Rahman, M. \textsl{Basic Hypergeometric Series}, Vol. 96, Cambridge University Press, 2011.
\bibitem{fab-rob-zwart} Guillemin, F., Robert, P., \& Zwart, B. (2004). AIMD algorithms and exponential functionals. \textsl{The Annals of Applied Probability}, 14(1): 90--117.
\bibitem{Haas} Haas, B. (2025). Precise asymptotics for the density and the upper tail of exponential functionals of subordinators. In: \emph{S\'eminaire de Probabiliti\'es LII, volume 2363} of Lecture notes in Mathematics, pages 389-426. Springer, Cham.
\bibitem{H26} Hu, D., Rachev, S.T., Sayit, H., Yang, H., \& Yildirim, Y. (2026). Iterated poisson processes for catastrophic risk modeling in ruin theory. \textsl{Insurance: Mathematics and Economics}, 126: 103200.
\bibitem{I93} Iserles, A. (1993). On the generalized pantograph functional-differential equation. \textsl{European Journal of Applied Mathematics}, 4: 1--38.
\bibitem{KT13} K\"{u}chler, U., \& Tappe, S. (2013). Tempered stable distributions and processes. \textsl{Stochastic Processes and Their Applications}, 123(12): 4256--4293.
\bibitem{K12} Kuznetsov, A. (2012). On the distribution of exponential functionals for L\'{e}vy processes with jumps of rational transform, \textsl{Stochastic Processes and Their Applications}, 122(2): 654--663.
\bibitem{KPS12} Kuznetsov, A., Pardo, J.C., \& Savov, M. (2012). Distributional properties of exponential functionals of L\'{e}vy processes. \textsl{Electronic Journal of Probability}, 17(8): 1--35.
\bibitem{M98} Mallik, R.K. (1998). Solutions of linear difference equations with variable coefficients. \textsl{Journal of Mathematical Analysis and Applications}, 222(1): 79--91.
\bibitem{MZ06} Maulik, K., \& Zwart, B. (2006). Tail asymptotics for exponential functionals of L\'evy processes. \textsl{Stochastic Processes and Their Applications}, 116(2): 156--177.
\bibitem{MS23} Minchev, M., \& Savov, M. (2023). Asymptotics for densities of exponential functionals of subordinators. \textsl{Bernoulli}, 29(4): 3307--3333.
\bibitem{Minchev-savov-2026} Minchev, M., \& Savov, M. (2026). Recent developments in exponential functionals of L\'evy processes. \textsl{Probability Surveys}, 23: 125--180.
\bibitem{OP12A} Orsingher, E., \& Polito, F. (2012). Compositions, random sums and continued random fractions of Poisson and fractional Poisson processes. \textsl{Journal of Statistical Physics}, 148(2): 233--249.
\bibitem{OT15} Orsingher, E., \& Toaldo, B. (2015). Counting processes with Bern{\v{s}}tein intertimes and random jumps. \textsl{Journal of Applied Probability}, 52(4): 1028--1044.
\bibitem{PRVS13} Pardo, J.C., Rivero, V., \& Van Schaik, K. (2013). On the density of exponential functionals of L\'{e}vy processes. \textsl{Bernoulli}, 19(5A): 1938--1964.
\bibitem{P12} Patie, P. (2012). Law of the absorption time of some positive self-similar Markov processes. \textsl{The Annals of Probability}, 40(2): 765--787.
\bibitem{PS18} Patie, P., \& Savov, M. (2018). Bernstein-gamma functions and exponential functionals of L\'{e}vy processes. \textsl{Electronic Journal of Probability}, 23(75): 1--101.
\bibitem{S11} Smith, H. \textsl{An Introduction to Delay Differential Equations with Applications to the Life Sciences}. Springer, 2011.
\bibitem{V20} Vostrikova, L. (2020). On distributions of exponential functionals of the processes with independent increments. \textsl{Modern Stochastics: Theory and Applications}, 7(3): 291--313.
\bibitem{X20} Xia, W. (2020). The average of a negative-binomial L\'{e}vy process and a class of Lerch distributions. \textsl{Communications in Statistics - Theory and Methods}, 49(4): 1008--1024.
\bibitem{X22} Xia, W. (2022). On inverse-power Poisson functionals. \textsl{Stochastics}, 94(1): 26--50.



\end{thebibliography}
\end{document}